\documentclass{amsart}
\RequirePackage[colorlinks,citecolor=blue,urlcolor=blue,linkcolor=blue]{hyperref}
\usepackage{natbib}  
\usepackage{color} 
\usepackage{hyperref} 
\usepackage{graphicx} 
\usepackage{geometry}
\usepackage[utf8]{inputenc}
\usepackage[T1]{fontenc}
\usepackage{verbatim}

\usepackage{tikz}
\usetikzlibrary{calc,shapes,backgrounds,arrows}
\usepackage{amssymb}

\renewcommand{\cite}{\citet}

\theoremstyle{plain}
\newtheorem{theorem}{Theorem}[section]
\newtheorem{proposition}[theorem]{Proposition}
\newtheorem{lemma}[theorem]{Lemma}

\theoremstyle{definition}

\theoremstyle{remark}
\newtheorem{remark}[theorem]{Remark}
\newtheorem{example}[theorem]{Example}

\usepackage{enumerate}

 \usepackage{framed,color,pgf}

 \usepackage{framed,pgf}
\newcounter{oldeq}
\newcounter{usesofarxiv}
 \newcommand{\arxiv}[1]{
\setcounter{oldeq}{\value{equation}}
 \addtocounter{usesofarxiv}{1}
 \setcounter{equation}{0}
\def\theoldeq{\theequation}
\def\theequation{x-\arabic{usesofarxiv}.\arabic{equation}}
\def\theequation{\arabic{section}.\arabic{usesofarxiv}.\arabic{equation}}
\def\theequation{\thesection.\arabic{usesofarxiv}.\arabic{equation}}
  \colorlet{shadecolor}{gray!10}
{
\begin{shaded}
\footnotesize
#1  \normalsize
\end{shaded}
   \setcounter{equation}{\value{oldeq}}
\numberwithin{equation}{section}
}}

 \usepackage{tikz}
 \usetikzlibrary{patterns}
 \usepackage{url}
\newcommand{\longcomment}[1]{}
\renewcommand{\longcomment}[1]{\ovalbox{\begin{minipage}{.9\textwidth}\color{blue}#1\end{minipage}}}

\def\fddto{\stackrel{\rm f.d.d.}{\Longrightarrow}}
\newcommand{\ind}{{\bf 1}}

\def\inddd#1{{\ind}_{\left\{#1\right\}}}

\newcommand{\esp}{{\mathbb E}}

\newcommand{\erfc}{{\rm erfc}}

\def\g{{\mathsf q}}
\def\p{{\mathsf p}}
\newcommand{\eqnh}{\begin{eqnarray*}}
\newcommand{\eqne}{\end{eqnarray*}}
\newcommand{\eqnhn}{\begin{eqnarray}}
\newcommand{\eqnen}{\end{eqnarray}}
\newcommand{\equh}{\begin{equation}}
\newcommand{\eque}{\end{equation}}

\def\summ#1#2#3{\sum_{#1 = #2}^{#3}}
\def\prodd#1#2#3{\prod_{#1 = #2}^{#3}}
\def\sif#1#2{\sum_{#1=#2}^\infty}

\def\topp#1{^{(#1)}}

\def\abs#1{\left|#1\right|}

\def\ccbb#1{\left\{#1\right\}}

\def\pp#1{\left(#1\right)}
\def\spp#1{(#1)}
\def\bb#1{\left[#1\right]}

\def\mmid{\;\middle\vert\;}

\def\floor#1{\left\lfloor #1 \right\rfloor}

\def\vv#1{{\boldsymbol #1}}
\def\vvalpha{{\vv\alpha}}
\def\vvbeta{{\vv\beta}}

\def\mand{\mbox{ and }}
\def\qmand{\quad\mbox{ and }\quad}

\def\qmwith{\quad\mbox{ with }\quad}
\def\mfa{\mbox{ for all }}

\def\wt#1{\widetilde{#1}}

\def\what#1{\widehat{#1}}

\def\R{{\mathbb R}}

\def\N{{\mathbb N}}

\def\fddto{\xrightarrow{\textit{f.d.d.}}}
\renewcommand{\ind}{{\bf 1}}

\def\inddd#1{{\ind}_{\left\{#1\right\}}}

\newcommand{\hidecomment}[1]{}
\newtheorem{assumption}{Assumption}[section]

\newcommand{\calE}{{\mathcal E}}
\newcommand{\calF}{{\mathcal F}}

\newcommand{\calL}{{\mathcal L}}

\newcommand{\calM}{\mathcal{M}}

\def\<{\langle}
\def\>{\rangle}

\usepackage{dsfont}
\newcommand{\RR}{\mathds{R}}

\newcommand{\ZZ}{\mathds{Z}}

\newcommand{\bfA}{\vv A}
\newcommand{\bfB}{\vv B}
\newcommand{\bfC}{\vv C}
\newcommand{\bfI}{\vv I}

\newcommand{\eps}{\varepsilon}

\usepackage{graphicx,fancybox,pgf,amsmath,float,amssymb}
\usepackage{fancybox,graphicx,dsfont,pgf}
\usepackage{subfigure,color,wrapfig}

\usepackage{latexsym}

\newcommand{\EE}{\mathds{E}}
\renewcommand{\Pr}{\mathds{P}}

\def\<{\langle}
\def\>{\rangle}

\def\A{{\mathsf a}}

\def\C{{ \mathsf c}}

\numberwithin{equation}{section}

\date{Revised  August 10, 2026}
\title[Fluctuations of random Motzkin paths]{Fluctuations of random Motzkin paths II}
\author{W\l odzimierz Bryc}
\address
{
W\l odzimierz Bryc\\
Department of Mathematical Sciences\\
University of Cincinnati\\
2815 Commons Way\\
Cincinnati, OH, 45221-0025, USA.
}
\email{wlodzimierz.bryc@uc.edu}
\urladdr{\href{https://homepages.uc.edu/~brycwz/}{https://homepages.uc.edu/~brycwz/}}
\author{Yizao Wang}
\address
{
Yizao Wang\\
Department of Mathematical Sciences\\
University of Cincinnati\\
2815 Commons Way\\
Cincinnati, OH, 45221-0025, USA.
}
\email{yizao.wang@uc.edu}
\urladdr{\href{http://homepages.uc.edu/~wangyz/}{http://homepages.uc.edu/~wangyz/}}

\begin{document}\sloppy

\begin{abstract}
We compute limit fluctuations of random Motzkin paths with arbitrary end-points as the length of the path tends to infinity.

\end{abstract}
\maketitle

\arxiv{This is an expanded version of the paper. It includes additional details, corrections, additional references, and it differs from the published version.}

\section{Introduction}

\subsection{Model and main result} %
A  Motzkin path of length $L$ is a sequence of steps on the integer lattice $\ZZ_{\geq 0}\times \ZZ_{\geq 0}$ that starts at point  $(0,n_0)$  with the initial altitude $n_0$ and ends at point $(L, n_L)$ at the final   altitude   $n_L$
for some non-negative integers $n_0,n_L,L$.
The steps can be up, down, or horizontal,   along the vectors $(1,1)$, $(1,-1)$ and $(1,0)$ respectively,
and the path cannot fall below the horizontal axis, see  \citet[Definition V.4, page 319]{flajolet09analytic}
 or \cite{Viennot-1984a}.
We
represent a Motzkin path of length $L\ge 1$ as a sequence of integers $(\gamma_0,\dots,\gamma_L)\in \ZZ_{\ge0}^{L+1}$ such that $|\gamma_k-\gamma_{k-1}|\le 1, k=1,\dots,L$ subject to the non-negativity condition   $\gamma_k\geq 0$ for $k=0,1\dots,L$. We say that the $k$-th step of the path is up, down, and horizontal respectively, if $\gamma_k-\gamma_{k-1} = 1, -1, 0$ respectively. By $\calM_{i,j}\topp L$ we denote the family of all Motzkin paths of length $L$ with the initial altitude $\gamma_0=i$ and the final altitude $\gamma_L=j$.
Our goal is to study statistical properties of random Motzkin paths, selected at random from the discrete set
\[
\calM\topp L=\bigcup_{i,j\geq 0} \calM_{i,j}\topp L
\] in the limit as $L\to\infty$.
Our setup generalizes  our previous work
\cite{Bryc-Wang-2019}, where we studied statistical properties of the three counting processes that count the up steps, the horizontal steps, and the down steps for a Motzkin path
$\gamma$ selected at random  uniformly from the set $\calM_{0,0}\topp L$.
To define these counting processes,
we  first introduce the indicators of these steps:
\begin{equation}\label{epsilons}
\eps_k^+\equiv   \eps_k^+(\gamma):=\inddd{\gamma_{k}>\gamma_{k-1}},\quad
\eps_k^-\equiv  \eps_k^-(\gamma):=\inddd{\gamma_{k}<\gamma_{k-1}}, \quad
\eps_k^0 \equiv \eps_k^0(\gamma):=\inddd{\gamma_{k}=\gamma_{k-1}}, \gamma\in \calM\topp L,
\end{equation}
and $k=1,\dots,L$.
For the sake of simplicity we drop the dependence on $\gamma$ of $\varepsilon$'s most of the time.
Then given a path $\gamma$ of length $L$, the counts  of the up steps, down steps and horizontal steps up to position $\floor{xL}$, where $x\in[0,1]$ are then
\begin{equation}\label{U-D-L}
  U_L(x):=\summ k1{\floor{ Lx}} \eps^+_k,\quad
   D_L(x):=\summ k1{\floor {Lx}} \eps^-_k,\quad
    H_L(x):=\summ k1{\floor{Lx}} \eps^0_k, \quad x\in[0,1].
\end{equation}

We introduce a probability measure on $\calM\topp L$ as follows.
For each path $\gamma\in\calM_{i,j}\topp L$,  we define its weight
\[%
  w_\sigma(\gamma):=\sigma^{\summ k1L\varepsilon_k^0}, \quad \gamma\in\calM_{i,j}\topp L, L\in\N.
\]%
Note that with $\sigma=1$ this gives each path the same weight.
Since $\calM_{i,j}\topp L$ is a finite set,
\[
  \mathfrak W_{i,j}\topp L=\sum_{\gamma\in\calM_{i,j}\topp L} w_\sigma(\gamma), i,j\ge 0,
\]
are well defined.
In addition to the weights of the edges, we also weight the initial and the final altitudes of each path with  geometric weights
\begin{equation}
  \label{geo}
  \alpha_{L,n} :=(\rho_{L,0})^n,\quad \beta_{L,n} :=(\rho_{L,1})^n, \quad n\ge 0,
\end{equation}
with
\begin{equation}
  \label{rho(L)}
\rho_{L,0}=1-\frac\C{\sqrt{L}}
\qmand \rho_{L,1}=1-\frac\A{\sqrt{L}} \quad \mbox{ for some } \A,\C\in\RR, \A+\C>0.
\end{equation}
Namely, the countable set  $\calM\topp L=\bigcup_{i,j\geq 0} \calM_{i,j}\topp L$ becomes a probability space with the discrete probability measure $\Pr_L$ determined by
\begin{equation}\label{Pr0}
\Pr_L(\gamma)\equiv\Pr_{\A,\C,\sigma,L}(\gamma) \equiv \Pr_{\A,\C,\sigma,L}(\{\gamma\}) = \frac{\alpha_{L,\gamma_0}\beta_{L,\gamma_L}}{\mathfrak{C}_{L}} w_\sigma(\gamma), \mfa \gamma\in \calM\topp L,
\end{equation}
with
\[   \mathfrak C_L :=\sum_{i,j\geq 0} \alpha_{L,i} \mathfrak W_{i,j}\topp L\beta_{L,j}<\infty.
\]
Note that throughout for finite $L$ implicitly we assume $L$ is large enough so that $\rho_{L,0}\rho_{L,1}\in(0,1)$ and hence $\Pr_L$ is a well-defined probability measure.
In our previous work \cite[Theorem 1.1]{Bryc-Wang-2019} we proved  that if $\gamma$ is selected uniformly from $\calM_{0,0}\topp L$, then
\begin{multline*} \frac{1}{\sqrt{2L}}\ccbb{U_L(x)-\frac{\floor{L x}}{3}, H_L(x)-\frac{\floor{L x}}{3},D_L(x)-\frac{\floor{L x}}{3}}_{x\in[0,1]}\\
\fddto
\ccbb{\frac{1}{2\sqrt{3}}B_x^{ex}+\frac{1}{6}B_x, -\frac{1}{3}B_x,  \frac{1}{6}B_x-\frac{1}{2\sqrt{3}}B_x^{ex}}_{x\in[0,1]}
, \end{multline*}
where $(B_x)_{x\in[0,1]}$ is a Brownian %
motion, $(B_x^{ex})_{x\in[0,1]}$ is a Brownian excursion, and the   processes $(B_x)_{x\in[0,1]}$ and $(B_x^{ex})_{x\in[0,1]}$ are independent. Formally, this model corresponds to the choice of $\rho_{L,0}=0, \rho_{L,1}=0,\sigma=1$.

Now, with more general end-point  weights that vary with $L$,  the asymptotics of \eqref{U-D-L} relies on another Markov process instead of the Brownian  excursion. Let
\begin{equation}
  \label{gg}
  \g_t(x,y) := \frac1{\sqrt {2\pi t}}\bb{\exp\pp{-\frac 1{2t}(x-y)^2} - \exp\pp{-\frac 1{2t}(x+y)^2}}\ind_{x>0,y>0}, \quad t>0,
\end{equation}
denote the transition kernel of the Brownian motion killed at hitting zero. Consider the Markov process $\pp{\widetilde\eta\topp{\A,\C}}_{x\in[0,1]}$  with joint probability density function at  points $0=x_0<x_1<\cdots<x_d=1$ given by %
\equh\label{eq:eta_pdf}
\wt p\topp{\A,\C}_{x_0,\dots,x_d}(y_0,\dots,y_d):= \frac{1}{\mathfrak C_{\A,\C}} e^{-(\C y_0+\A y_d)/{\sqrt 2}}\,\prod_{k=1}^d \g_{x_k-x_{k-1}}(y_{k-1},y_k),\quad y_0,\dots,y_d>0,
\eque
 with the normalizing constant
\begin{equation}
  \label{C(a,c)}
   \mathfrak C_{\A,\C}= \int_{\RR_+^2} e^{-(\C x+\A y)/{\sqrt 2}}\g_1(x,y)dx dy,
\end{equation}
 given by the explicit expression \eqref{eq:C_ac}.
Let  $\eta\topp{\A,\C}$ denote the increment process
\equh\label{eq:eta}
\eta\topp{\A,\C}_x:=\widetilde\eta\topp{\A,\C}_x-\widetilde\eta\topp{\A,\C}_0, x\in[0,1].
\eque

Recall that  for each $L$ fixed we let $(\gamma_0,\dots,\gamma_L)$ denote a sequence from $\calM\topp L$ sampled from $\Pr_L$ given in \eqref{Pr0}, including in particular the left-hand side of \eqref{path}, and the counting processes $U_L,H_L,D_L$ depend on $(\gamma_0,\dots,\gamma_L)$  as in \eqref{U-D-L}. Our main result is the following.
\begin{theorem}\label{thm:1}
Assume $\A,\C\in\R, \A+\C>0$ and $\sigma>0$. Set
\begin{equation}
  \label{a'c'} \A'=\frac{2\A}{\sqrt{2+\sigma}}, \quad \C'=\frac{2\C}{\sqrt{2+\sigma}}.
\end{equation}
Then the following convergence holds.
\begin{enumerate}[(i)]
\item  %
As $L\to\infty$, we have %
\begin{equation}
  \label{path}
  \sqrt{\frac{2+\sigma}{2 L}}\left(\gamma_{\floor{Lx}} \right)_{x\in[0,1]}\fddto
\pp{\wt \eta_x\topp {\A',\C'}}_{x\in[0,1]}.
\end{equation}
  \item   %
As $L\to\infty$, we have %
\begin{multline}
  \label{UHD}
\frac1{\sqrt {2L}}\,\ccbb{U_L(x)-\frac{1}{2+\sigma}\floor{Lx},\;  H_L(x)-\frac{\sigma }{2+\sigma}\floor{Lx},\;D_L(x)-\frac{1}{2+\sigma}\floor{Lx}}_{x\in[0,1]}\\
\fddto\ccbb{ \frac1{2\sqrt{2+\sigma}}\eta_x\topp {\A',\C'}+{\frac{\sqrt\sigma}{2(2+\sigma)}} B_x,\; - {\frac{\sqrt\sigma}{2+\sigma}} B_x, \; {\frac{\sqrt\sigma}{2(2+\sigma)}} B_x- \frac1{2\sqrt{2+\sigma}}\eta_x\topp {\A',\C'}}_{x\in[0,1]}
, \end{multline}
where $(B_x)_{x\in[0,1]}$ is a Brownian %
motion, $(\eta_x \topp {\A',\C'})_{x\in[0,1]}$ is given by \eqref{eq:eta}, and the   processes $(B_x)_{x\in[0,1]}$ and $(\eta\topp{\A',\C'}_x)_{x\in[0,1]}$ are independent.

\end{enumerate}
\end{theorem}
\begin{remark}\label{Rem.Don}
Note that as a corollary of Theorem \ref{thm:1}(ii), using $U_L(x)-D_L(x)=\gamma_{\floor{L x}}-\gamma_0$,  we have
 $$
     \frac{1}{\sqrt{2L}}\{\gamma_{\floor{L x}}-\gamma_0\}_{x\in[0,1]}\fddto \frac{1}{\sqrt{2+\sigma}}\eta\topp{\A',\C'} \quad \mbox{ as $L\to\infty$}.$$
This result, in fact, can be obtained directly by a soft argument and in a stronger convergence mode, as shown in the next proposition. We thank an anonymous referee for this observation.
 It is plausible that convergence in \eqref{path} and \eqref{UHD} can also be strengthened to convergence in %
  $D[0,1]$.
\end{remark}
\begin{proposition}\label{prop:1}
 Under the assumptions of Theorem \ref{thm:1},
with
\[
\xi_L(x):=\gamma_{\floor{L x}}-\gamma_0, x\in[0,1],
\]
we have
\begin{equation}
  \label{BD-Don}
\ccbb{  \frac1{\sqrt{L}}\xi_L(x)}_{x\in[0,1]}\Rightarrow \ccbb{\sqrt{\frac2{2+\sigma}} \eta_x\topp{\A',\C'}}_{x\in[0,1]},
\end{equation}
as $L\to\infty$
in Skorohod's space of c\`adl\`ag functions $D[0,1]$.
\end{proposition}
\begin{remark}
One can also work with random Motzkin paths with fixed left-end point zero and geometric weights for the right-end point, and obtain a corresponding joint convergence with a `randomized' Brownian meander (with joint probability density function proportional to
 $$e^{-\A' y_d/\sqrt 2}\ell_{x_1}(y_1)\prod_{k=2}^d \g_{x_k-x_{k-1}}(y_{k-1},y_k),  \quad \ell_t(y)=\frac{y}{\sqrt{2\pi t^3}} e^{-y^2/(2t)}$$ for $\A'>0$) in place of $B^{ex}$ and $\eta\topp{\A',\C'}$ above. Both Brownian excursion and randomized Brownian meanders showed up already in \citet{Bryc-Wang-Wesolowski-2022} in the study of limit fluctuations of height functions for open ASEP, denoted by $\eta\topp{\infty,\infty},\eta\topp{\A',\infty}$ therein. We omit the details for this case.
\end{remark}

\subsection{Motivation}
Process $(\eta\topp{\A,\C}_x)_{x\in[0,1]}$ from \eqref{eq:eta} has recently appeared in
investigations of non-equilibrium systems in the mathematical physics literature.

 First, it was shown in \citet{barraquand2022steady} and \citet{Bryc-Kuznetsov-2021} that $\eta\topp{\A,\C}$ can be obtained as a re-scaling of the processes
 that appeared %
  in the description
  of the {\em stationary measure of open KPZ} (on an interval), recently identified in \citet{corwin21stationary}, \citet{bryc23markov} and \citet{barraquand2022steady}. Namely, one can represent
 the stationary measure of the open KPZ equation  on  an interval $[0,\tau]$
  as
  \begin{equation}
    \label{CK-21}
    \left\{\frac{1}{\sqrt{2}}B_x\right\}_{x\in[0,\tau]}+\left\{Y\topp{\A,\C}_x - Y\topp{\A,\C}_0\right\}_{x\in[0,\tau]},
  \end{equation} where $B$ is a  Brownian motion, and processes $B$ and $Y$ are independent.
   As $\tau\to\infty$, we then have
\[%
\ccbb{\frac1{\sqrt{\tau}}Y\topp{\A/\sqrt\tau, \C/\sqrt\tau}_{x\tau}}_{x\in[0,1]} \fddto \ccbb{\frac1{\sqrt 2}\wt\eta_x\topp{\A,\C}}_{x\in[0,1]},
\]%
and hence
\equh\label{eq:second'}
\ccbb{\frac1{\sqrt{\tau}}\pp{Y\topp{\A/\sqrt\tau, \C/\sqrt\tau}_{x\tau} - Y_0\topp{\A/\sqrt\tau,\C/\sqrt\tau}}}_{x\in[0,1]} \fddto \ccbb{\frac1{\sqrt 2}\eta_x\topp{\A,\C}}_{x\in[0,1]}
\quad \mbox{ as $\tau\to\infty$.}
\eque
(The process denoted by $\wt \eta$ in \cite[Theorem 2.1]{Bryc-Kuznetsov-2021}
is $\frac1{\sqrt{2}}\widetilde\eta\topp{\A,\C}$
 here.)
The identification of the process $Y\topp{\A,\C}$ is a recent groundbreaking work. It is a Markov process with transitional law determined by  Doob's $h$-transform applied to the Yakubovich heat kernel; see \cite{bryc23markov} for details. The process \eqref{CK-21}
arises
in the scaling limit of height function of particle densities of open ASEP with five parameters
$\alpha_n,\beta_n,\gamma_n,\delta_n,q_n$
 all depending on the %
 size $n$
 of the system and appropriately chosen to satisfy the Liggett's condition.
It was conjectured by \citet{barraquand2022steady} that $\eta\topp{\A,\C}$ appears in the description of the stationary measure of the open KPZ fixed point, a space-time Markov process that has not been rigorously defined yet in the literature.
Note that the limit theorem %
\eqref{eq:second'}
 leading to $\eta\topp{\A,\C}$ as summarized above can be understood as a double-limit theorem (first the convergence from the height function of open ASEP to $\{Y_x\topp{\A,\C}-Y_0\topp{\A,\C}\}_{x\in[0,\tau]}$, and then the second convergence \eqref{eq:second'}).

Second, it was later shown by \citet{Bryc-Wang-Wesolowski-2022} that
with parameters $\alpha_n,\beta_n,\gamma_n,\delta_n$ appropriately chosen and $q\in[0,1)$ fixed,
the process %
$(\eta\topp{\A,\C}+ B)/\sqrt{2}$, where $B$ is an independent Brownian motion, %
arises directly as the scaling limit of the height function of particle densities.
This convergence, in contrast to the first case,  can be understood as a single-limit theorem. %

The contribution of this paper is a third limit theorem for the process $\wt\eta\topp{\A,\C}$. We show that this process arises as the scaling limit of random Motzkin paths. Our model and analysis are considerably simpler than the open ASEP, and therefore the limit theorem provides  quick access to the process $\wt\eta\topp{\A,\C}$.
At the same time, we emphasize that we focus on the stationary measure
of the conjectured open KPZ fixed point,
instead of the dynamics of the model (say starting from an arbitrary initial configuration).

The paper is organized as follows.
Section \ref{sec:representations} provides matrix and Markov representations for a larger class of random Motzkin paths including the one in Theorem \ref{thm:1}. Section \ref{sec:proof} provides the proof of Theorem \ref{thm:1}. Section \ref{sec:prop1} provides the proof of Proposition \ref{prop:1}.

\section{Matrix and Markov representations for random Motzkin paths}\label{sec:representations}
Our method is based on the fact that explicit integral representations of statistics of interest are available in closed form, and moreover they are convenient for asymptotic analysis.
We shall establish these representations for a larger class of random Motzkin paths than those considered in Theorem \ref{thm:1} (which corresponds to taking $\vv a= \vv c = (1,1,\dots)$ and $\vv b = (\sigma,\sigma,\dots)$ below).

Throughout this section, the length of the Motzkin paths $L$ is fixed.
  We first construct the weights of edges from three sequences
  \[
  \vv a=(a_j)_{j\geq 0},\quad \vv b=(b_j)_{j\geq 0},\quad \vv c=(c_j)_{j\geq 1},
  \]
  of real numbers, where we assume that  $a_0,a_1,\ldots>0$, $b_0,b_1,\ldots \geq 0$, and   $c_1,c_2,\ldots > 0$.
For each path
\[
\gamma=(\gamma_0=i,\gamma_1,\dots,\gamma_{L-1},\gamma_L=j)\in\calM_{i,j}\topp L,
\] we define its weight
\[%
  w(\gamma)\equiv w_{\vv a,\vv b, \vv c,L}(\gamma)=\prod_{k=1}^L a_{\gamma_{k-1}}^{\eps_k^+} b_{\gamma_{k-1}}^{\eps_k^0}c_{\gamma_{k-1}}^{\eps_k^-}, \quad \gamma\in\calM_{i,j}\topp L, L\in\N.
\]%
That is, we take $\vv a$, $\vv b$ and $\vv c$ as the weights of the up steps, horizontal steps and down steps, and the weight of a step depends also on the altitude of the left-end of an edge.
Since $\calM_{i,j}\topp L$ is a finite set, the normalization constants
\[
  \mathfrak W_{i,j}\topp L=\sum_{\gamma\in\calM_{i,j}\topp L} w(\gamma)
\]
are well defined for all $i,j\geq 0$.

In addition to the weights of the edges, we wish to also weight the end-points, i.e. the initial and the final altitudes of a Motzkin path.
To this end we choose two additional non-negative sequences $\vv\alpha=(\alpha_i)_{i\ge 0}$ and $\vv\beta=(\beta_i)_{i\ge 0}$ such that

\begin{equation}
   \label{CL}
   \mathfrak C_L \equiv \mathfrak{C}_{\vv\alpha,\vv\beta,\vv a, \vv b, \vv c, L}:=\sum_{i,j\geq 0} \alpha_i \mathfrak W_{i,j}\topp L\beta_j<\infty.
\end{equation}
Most of the time, for the sake of simplicity we drop the dependence on the boundary-weight parameters $\vvalpha,\vvbeta$ and edge-weight parameters  $\vv a,\vv b, \vv c$, but keep the dependence on the length $L$.

Note that   $\mathfrak W_{i,j}\topp L=0$ for $|j-i|>L$. So if the sequences $\vv a,\vv b, \vv c$ are bounded, then $\mathfrak W_{i,j}\topp L$ are also bounded, and \eqref{CL} is finite if $\sum_{n\ge 0} \alpha_n\beta_{n+j}<\infty$ for $-L\le j\le L$. (Here, $\beta_j=0$ for $j<0$.)
With finite normalizing constant \eqref{CL}, the countable set  $\calM\topp L=\bigcup_{i,j\geq 0} \calM_{i,j}\topp L$ becomes a probability space with the discrete probability measure
\[%
\Pr_L(\gamma)\equiv  \Pr_{\vv\alpha,\vv\beta,\vv a, \vv b, \vv c, L}(\gamma)=\frac{\alpha_{\gamma_0}\beta_{\gamma_L}}{\mathfrak{C}_{L}} w(\gamma), \mfa \gamma\in \calM\topp L.
\]%
By a random Motzkin path of length $L$, we refer to the random element in $\calM\topp L$ with law $\Pr_L$.

Such a construction seems to be  folklore.  The case $\alpha=(1,0,0,\dots),\beta=(1,0,0,\dots)$ and $\vv a=\vv b=(1,1,\dots)$, $\vv c=(1,1,\dots)$   recovers the uniform choice of Motzkin paths from $\calM_{0,0}\topp L$ that  we considered in \cite[Theorem 1.1]{Bryc-Wang-2019}.
Of our special interest is the example with  bounded $\vv a,\vv b,\vv c$ and geometric weights
\equh\label{eq:geometric}
\alpha_n=\rho_0^n, \beta_n=\rho_1^n, n\ge 0, \quad\mbox{ for some }\quad \rho_0,\rho_1>0, \rho_0\rho_1<1.
\eque
 In this case, the normalizing constant \eqref{CL} is finite when the product $\rho_0\rho_1<1$.

For non-uniform laws and geometric boundary weights, we mention
an example
that motivated our framework here.
\begin{example}
 \citet[Section V.4]{flajolet09analytic}
 and \cite{Viennot-1984a}  consider the case of equal weights $\vv a=(1,1,\dots)$ for the up-steps, with varying weights of horizontal and down steps. The choice
\[
\alpha_n=\pp{\frac{1-\alpha}{\alpha}}^{n+1},\; \beta_n=\pp{\frac{1-\beta}{\beta}}^{n+1},\; \vv a=\vv c=(1,1,\dots),\; \vv b=(2,2,\dots),
\]
with $\alpha,\beta\in(0,1)$ such that $\alpha+\beta>1$ recovers Motzkin paths that appear in the analysis of open TASEP in \cite[Section 2.2]{derrida04asymmetric} (after shifting their paths  down by one unit).
\end{example}

  \begin{figure}[H]
  \begin{tikzpicture}[scale=1]
 \draw[->] (0,0) to (0,3);

 \draw[->] (0,0) to (11,0);
\draw[-,thick] (0,0) to (1,0);
\draw [fill] (0,0) circle [radius=0.05];
 \node[above] at (.6,0) {\footnotesize  $b_0$};
\draw[-,thick] (1,0) to (2,1);
\draw [fill] (1,0) circle [radius=0.05];
 \node[above] at (1.4,0.6) {\footnotesize  $a_0$};
\draw[-,thick] (2,1) to (3,0);
\draw [fill] (2,1) circle [radius=0.05];
 \node[above] at (2.6,0.6) {\footnotesize  $c_1$};
\draw[-,thick] (3,0) to (4,1);
\draw [fill] (3,0) circle [radius=0.05];
 \node[above] at (3.4,0.6) {\footnotesize  $a_0$};
\draw[-,thick] (4,1) to (5,1);
\draw [fill] (4,1) circle [radius=0.05];
 \node[above] at (4.5,1) {\footnotesize  $b_1$};
\draw[-,thick] (5,1) to (6,2);
\draw [fill] (5,1) circle [radius=0.05];
 \node[above] at (5.5,1.6) {\footnotesize  $a_1$};
\draw[-,thick] (6,2) to (7,1);
\draw [fill] (6,2) circle [radius=0.05];
 \node[above] at (6.5,1.6) {\footnotesize  $c_2$};
\draw[-,thick] (7,1) to (8,0);
\draw [fill] (7,1) circle [radius=0.05];
 \node[above] at (7.5,0.6) {\footnotesize  $c_1$};
\draw[-,thick] (8,0) to (9,1);
\draw [fill] (8,0) circle [radius=0.05];
 \node[above] at (8.4,0.6) {\footnotesize  $a_0$};
 \draw [fill] (9,1) circle [radius=0.05];

   \node[below] at (1,0) {  $1$};
      \node[below] at (2,0) {  $2$};
       \node[below] at (3,0) {  $3$};
        \node[below] at (4,0) {  $4$};
         \node[below] at (5,0) {  $5$};
          \node[below] at (6,0) {  $6$};
           \node[below] at (7,0) {  $7$};
            \node[below] at (8,0) {  $8$};
             \node[below] at (9,0) {  $9$};

\end{tikzpicture}
\caption{Motzkin path $\gamma=(0,0,1,0,1,1,2,1,0,1)\in\calM\topp 9$     with weight contributions marked at the edges.  The probability of selecting this path from $\calM\topp 9$ is $\Pr(\gamma)=\frac{\alpha_0 \beta_1}{\mathfrak C_9} b_0 b_1 a_0^3 a_1 c_1^2  c_2$. The total number of horizontal steps is $H_{9}(1)=2$ and the total number of up steps is $U_{9}(1)=4$.}
\label{Fig2.1}
\end{figure}
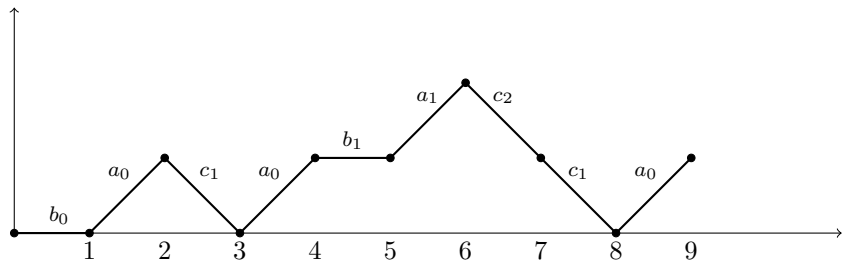

Recall that the general framework of the random Motzkin paths depends on the edge-weight parameters $\vv a,\vv b,\vv c$ and boundary-weight parameters $\vv\alpha,\vv\beta$. For such a random Motzkin path with length $L$, we let
\[
\Pr_L\equiv \Pr_{\vvalpha,\vvbeta, \vv a,\vv b, \vv c, L}
\] denote its law (a probability measure on $\calM\topp L$), and $\EE_L$ the expectation with respect to $\Pr_L$.

\subsection{Matrix representation}
We first start with a matrix representation, known as the matrix ansatz in the literature. Introduce
$$\bfA=
\begin{bmatrix}
  0& a_0 & 0 &0 &\dots \\
  0& 0 &a_1 & 0& \\
  0&0&0&a_2 & \\
  \vdots& & & &\ddots
\end{bmatrix},
\;
\bfB=
\begin{bmatrix}
  b_0& 0 & 0 &0 &\dots \\
  0& b_1 &0 & 0& \\
  0&0&b_2&0 & \\
    \vdots& & & \ddots
\end{bmatrix},
\;
\bfC=
\begin{bmatrix}
  0& 0 & 0 &0 &\dots \\
  c_1& 0 &0 & 0& \\
  0&c_2&0&0 & \\
    \vdots& & \ddots
\end{bmatrix}.
$$
Furthermore, introduce two vectors
$$
\langle W_{\vvalpha}(z) | = \begin{bmatrix}
 \alpha_0&\alpha_1z& \alpha_2 z^2&\dots
\end{bmatrix}, \quad | V_{\vvbeta}(z) \rangle =
\begin{bmatrix}\beta_0\\\beta_1z\\ \beta_2 z^2\\\vdots\end{bmatrix},
$$
which are viewed as functions of $z$.
Recall the ``decomposition" of a Motzkin path defined in \eqref{epsilons}.
Throughout, for a product of matrices $M_1,\dots,M_L$, we take the convention $\prod_{k=1}^L M_k = M_1M_2\cdots M_L$.

\begin{lemma}

Under assumption \eqref{CL},  given $s_k,t_k,u_k>0$ and $z_0,z_1\in(0,1]$, we have
\begin{align}
\sum_{\gamma\in\calM\topp L}z_0^{\gamma_0}\prod_{k=1}^L \pp{s_k^{\eps_k^+}t_k^{\eps_k^-}u_k^{\eps_k^0} }z_1^{\gamma_L}\alpha_{\gamma_0}w(\gamma)\beta_{\gamma_L} & = \left\langle W_\vvalpha(z_0) \middle|\prod_{k=1}^L (s_k \bfA+t_k\bfC+u_k\bfB)
     \middle| V_\vvbeta(z_1) \right\rangle,\nonumber\\
\mathfrak C_L = \sum_{\gamma\in\calM\topp L}\alpha_{\gamma_0}w(\gamma)\beta_{\gamma_L}& = \langle W_\vvalpha(1) |(\bfA+\bfC+\bfB)^L
     | V_\vvbeta(1) \rangle. \label{eq:CL2}
     \end{align}
In particular,
\begin{equation}\label{matrix-ansatz}
     \EE_L\left[ z_0^{\gamma_0}\prod_{k=1}^L s_k^{\eps_k^+}t_k^{\eps_k^-}u_k^{\eps_k^0} z_1^{\gamma_L}\right]
     =
     \frac{1}{\mathfrak C_L}\left\langle W_\vvalpha(z_0) \middle|\prod_{k=1}^L (s_k \bfA+t_k\bfC+u_k\bfB)
     \middle| V_\vvbeta(z_1) \right\rangle.
\end{equation}
\end{lemma}
\begin{proof}
We first notice that, by definition,
\equh\label{eq:EL}
     \EE_L\left[ z_0^{\gamma_0}\prod_{k=1}^L s_k^{\eps_k^+}t_k^{\eps_k^-}u_k^{\eps_k^0} z_1^{\gamma_L}\right]
 = \frac{\sum_{i,j\geq 0} \alpha_iz_0^i\beta_jz_1^j\sum_{\gamma\in\calM_{i,j}\topp L} \prod_{k=1}^L (s_k^{\eps_k^+}t_k^{\eps_k^-}u_k^{\eps_k^0})  w(\gamma)}{\sum_{i,j\geq 0} \alpha_i\beta_j\sum_{\gamma\in\calM_{i,j}\topp L} w(\gamma)},
\eque
and the denominator on the right-hand side is nothing but $\mathfrak C_{L}$ in \eqref{CL}. Recall that $\varepsilon$'s are functions of $\gamma$.

We start by proving the formula for $\mathfrak C_L$. First recall the following well-known fact. Consider a finite (say $n$) state Markov chain. Let $P = (P_{i,j})_{i,j=1,\dots,n}$ be its transitional probability matrix, so that $P_{i,j}$ is the probability
of transitioning
from state $i$ to state $j$ in one step. Let $\vec\pi$, a vertical vector in $\R^n$, represent a marginal law of the Markov chain. Then, $\vec\pi^T P^k$ represents the marginal law of the Markov chain starting from the law represented by $\vec\pi$ in $k$ steps. This representation can be extended to a  Markov chain with countably infinite states, and also to inhomogeneous ones.

Moreover, this presentation can be further extended to the situation where $P$ is replaced by a weight matrix (each entry is non-negative but the sum of each row is not necessarily one), and where the sum of entries in $\vec\pi$ is not necessarily one. In this case, $\vec\pi$ and $\vec\pi^T P^k$ are no longer interpreted as probability laws. However, by the proof behind the interpretation of $\vec\pi^T P^k$ above,  it is readily checked that $(\vec\pi^T P^k)_i$ is the total weight of all paths (understood in the obvious way) ending at location $i$ in $k$ steps, with unit weight assigned to the location $i$.

The above discussion provides an interpretation of $\langle W_\vvalpha(1) |(\bfA+\bfC+\bfB)^L
     |$, with $P = \bfA+\bfC+\bfB$ and $\vec\pi^T = \langle W_\vvalpha(1)|$, as the total weights of $L$-step paths with initial weight $\vec\pi$ and weight matrix $P$, and {\em uniform weight on the end points}.   Now, the right-hand side of \eqref{eq:CL2}, with the extra factor $|V_\vvbeta(1)\rangle$ on the right, can be interpreted similarly, with weights $\vv\beta$ assigned additionally to the end locations. Therefore \eqref{eq:CL2} follows.

     For the numerator on the right-hand side of \eqref{eq:EL}, notice that one can write
     \[
     \sum_{\gamma\in\calM_{i,j}\topp L} \prod_{k=1}^L s_k^{\eps_k^+}t_k^{\eps_k^-}u_k^{\eps_k^0}  w(\gamma) =
          \sum_{\gamma\in\calM_{i,j}\topp L}  \wt w(\gamma) \qmwith \wt w(\gamma) = \prodd k1L (s_ka_{\gamma_{k-1}})^{\varepsilon_k^+}(u_kb_{\gamma_{k-1}})^{\varepsilon_k^0}(t_k c_{\gamma_{k-1}})^{\varepsilon_k^-}.
     \]
     So, again by the same interpretation as before but now for inhomogeneous weight matrices $(s_k\bfA+t_k\bfC+u_k\bfB)_{k=1,\dots,L}$, we see
     \[
     \left\langle W_\vvalpha(z_0) \middle|\prod_{k=1}^L (s_k \bfA+t_k\bfC+u_k\bfB)
     \middle| V_\vvbeta(z_1) \right\rangle = \sum_{i,j\geq 0} \alpha_iz_0^i\beta_jz_1^j\sum_{\gamma\in\calM_{i,j}\topp L}\wt w(\gamma).
     \]
     This completes the proof.\end{proof}
The left-hand side of \eqref{matrix-ansatz} can be related to the joint Laplace transform of finite-dimensional distributions of the random Motzkin paths. However, the matrix representation on the right-hand side is not always convenient for asymptotic analysis.

\subsection{Markov representation}
The next step is to re-express the matrix representation in terms of integrals (expectations) involving  a certain Markov process.
Also in this step, we eliminate one of the 3 variables by the relation  $\eps_k^++\eps_k^-+\eps_k^0=1$.
That is, we shall be interested here in \eqref{matrix-ansatz} with $s_k=1$.

First, for $t>0$ consider a family of orthogonal polynomials $\{p_n(x;t)\}_{n\ge 0}$ with the  Jacobi matrix $\bfA+t \bfC$. That is,
with%
 $$\vec{\vv p}(x;t)=\begin{bmatrix}
     p_0(x;t) \\ p_1(x;t)\\p_2(x;t) \\ \vdots
   \end{bmatrix},
   $$
the orthogonal polynomials are determined by
\[%
x\vec{\vv p}(x;t) = (\bfA+t\bfC)\vec{\vv p}(x;t), t>0,
\]%
or equivalently,
\begin{equation}\label{p-rec}
  xp_n(x;t) = a_n p_{n+1}(x;t) + t c_n p_{n-1}(x;t), \; n\ge 0,
\end{equation}
with $p_0(x;t) = 1, p_{-1}(x;t) = 0$.
For each $t>0$ let $\nu_t$ denote the associated orthogonal measure.
\arxiv{Note that  for $n=0$ we have $x p_0(x;t)=a_0 p_1(x;t)$, so the value of $c_0$ is irrelevant and can be taken as $0$.}
\begin{assumption}\label{assumption:1}
Consider $\vec{\vv p}$ and $\{\nu_t\}_{t> 0}$ as above for $\bfA$ and $\bfC$ given. We assume that there exists a Markov process $(X_t)_{t>0}$ such that the law of $X_t$ is $\nu_t$ and furthermore that for each $n\ge 0$, the stochastic process $\{p_n(X_t;t)\}_{t>0}$ is a martingale polynomial in the sense that
\[
\EE(p_n(X_t;t)|X_s)=p_{n}(X_s;s) \mfa 0< s\le t.
\]
\end{assumption}
Some general conditions  in terms of matrices $\bfA$ and $\bf C$  for the existence of such a Markov process could be read off from \cite{Bryc-Matysiak-Wesolowski-04}; there are also many classical as well as less-classical examples, see e.g.
\citet{Bryc-Wesolowski-03,Bryc-Wesolowski-08}.
\arxiv{%
(Note that the three-step recurrence \eqref{p-rec}  does not have the middle term, i.e. it corresponds to $\bfB=0$. We will, however, be able to use these polynomials to analyze Motzkin paths with the weight $w(\gamma)$ arising from $\bfB=\sigma\bfI$.)}
It is easy to check that with
\[
p_n(x):=p_n(x;1),
\]
the solution of the three-step recurrence \eqref{p-rec}  is
\begin{equation}
  \label{pt}
  p_n(x;t)= t^{n/2}p_n(x /\sqrt{t}), t>0.
\end{equation}
So measure $\nu_t$ is just a dilation of measure $\nu\equiv \nu_1$, in the sense that
$\nu_t(A)=\nu(A/\sqrt{t})$.
It is also well known (\cite[(1.23)]{Askey-Wilson-1985}) that
\[
\|p_n(\cdot;t)\|_{L^2(\nu_t)}^2 :=\int_\R p_n^2(x;t) \nu_t(dx)= \prod_{k=1}^n \frac{tc_k}{a_{k-1}}=t^n \|p_n\|_{L^2(\nu)}^2.
\]

We next introduce two generating functions
\[%
\phi_\vvalpha(x,z)= \sum_{n=0}^\infty \alpha_n z^n p_n(x)\qmand
\psi_\vvbeta(x,z)=\sum_{n=0}^\infty \beta_n z^n \frac{p_n(x)}{\|p_n\|_{L^2(\nu)}^2}.
\]%
We assume that  $\nu$ has compact support and that
both series converge absolutely, uniformly in $x$ on the support of  $\nu$, provided $|z|<z_\vvalpha$ for the series $\phi_\vvalpha(x,z)$ and $|z|<z_\vvbeta$ for the series $\psi_\vvbeta(x,z)$.

In view of \eqref{pt}, we have
\[%
\sum_{n=0}^\infty \alpha_n z^n p_n(x;t)=\phi_\vvalpha\pp{\frac x{\sqrt{t}},\ z{\sqrt{t}}} \qmand
 \sum_{n=0}^\infty \beta_n z^n \frac{p_n(x;t)}{\|p_n(\cdot;t)\|_{L^2(\nu_t)}^2}=\psi_\vvbeta\pp{\frac x{\sqrt{t}},\frac z{\sqrt{t}}}.
\]%

 \begin{proposition}\label{prop:markov}
 Consider fixed parameters
 \[
  t_1\geq t_2\geq \dots\geq t_L>0, \quad u_1,\dots,u_L>0, \mand z_0,z_1 \mbox{such that } |z_0|^2t_1<z_\vvalpha^2,|z_1|^2/t_L<z_\vvbeta^2.
  \]
 {  With $\bfB=\sigma\bfI$, }
  we have %
\begin{equation}\label{Mkv-ans0}
 \sum_{\gamma\in\calM\topp L} \alpha_{\gamma_0}  z_0^{\gamma_0}\beta_{\gamma_L} z_1^{\gamma_L} w(\gamma)\prod_{k=1}^L t_k^{\eps_k^-}u_k^{\eps_k^0}
     ={\EE\left[ \phi_\vvalpha\pp{\frac{X_{t_1}}{\sqrt{t_1}},z_0\sqrt{t_1 }}\psi_\vvbeta\pp{\frac{X_{t_L}}{\sqrt{t_L}},\frac{z_1}{\sqrt{t_L}}} \prod_{k=1}^L(\sigma u_k+X_{t_k}) \right]}.
 \end{equation}
 In particular,  %
   \begin{equation}\label{markov-ansatz}
     \EE_L \left[z_0^{\gamma_0}z_1^{\gamma_L}\prod_{k=1}^L t_k^{\eps_k^-}u_k^{\eps_k^0}
    \right] =
     \frac{1}{\mathfrak C_L} \EE\left[ \phi_\vvalpha(X_{t_1}/\sqrt{t_1},z_0\sqrt{t_1})\psi_\vvbeta(X_{t_L}/\sqrt{t_L},z_1/\sqrt{t_L}) \prod_{k=1}^L(\sigma u_k+X_{t_k}) \right].
\end{equation}
Here $\mathfrak C_L$ is the normalization constant \eqref{CL}.
\end{proposition}

The proof is based on the ideas in  \cite{Bryc-Wesolowski-08}, but there are also
significant differences. In particular we do not rely on $q$-commutation equations or quadratic harnesses.

 \begin{proof}
 Denote
  $$\vec{\vv p}(x;t)=\begin{bmatrix}
     p_0(x;t) \\ p_1(x;t)\\p_2(x;t) \\ \vdots
   \end{bmatrix}.$$

First, notice that by orthogonality %
 we have
\[
 |V_\vvbeta(z_1)\rangle =
 \begin{bmatrix}
     \beta_0 \\ \beta_1z_1\\\beta_2z_1^2 \\ \vdots
   \end{bmatrix} = \EE\bb{\sif n0 \beta_nz_1^n\frac{p_n(X_{t_L};t_L)}{{\|p_n(\cdot;t_L)\|_{L^2(\nu_{t_L})}^2}}\vec {\vv p}(X_{t_L};t_L)} =
   \EE\bb{\psi_\vvbeta\pp{\frac{X_{t_L}}{\sqrt{t_L}},\frac{z_1}{\sqrt{t_L}}}\vec{\vv p}(X_{t_L};t_L)}.
\]
\arxiv{Indeed,
$
  \psi_\vvbeta\pp{\frac{X_{t_L}}{\sqrt{t_L}},\frac{z_1}{\sqrt{t_L}}}= \sum_{n=0}^\infty \beta_n z_1^n \frac{p_n(X_{t_L};t_L)}{\|p_n(\cdot;t_L)\|_{L^2(\nu_{t_L})}^2}
$ converges in $L_2$, so on the $k$-th coordinate $$\EE[\psi_\vvbeta\pp{\frac{X_{t_L}}{\sqrt{t_L}},\frac{z_1}{\sqrt{t_L}}}p_k(X_{t_L};t_L)]=
\sum_{n=0}^\infty \beta_n z_1^n \frac{\EE[p_n(X_{t_L};t_L)p_k(X_{t_L};t_L)]}{\|p_n(\cdot;t_L)\|_{L^2(\nu_{t_L})}^2}
=\beta_k z_1^k.$$
}
 Note also that
 \[
 (\bfA+u \bfB+ t \bfC)\vec{\vv p}(x;t)=(x +\sigma u)\vec{\vv p}(x;t).
 \]
  So
 \begin{align*}
    (\bfA+t_L\bfC+u_L\bfB)|V_\vvbeta(z_1)\rangle
&  =  \EE \left[ \psi_\vvbeta(X_{t_L}/\sqrt{t_L},z_1/\sqrt{t_L})(\bfA+t_L\bfC+u_L \bfB)\vec{\vv p}(X_{t_L};t_L)\right]
    \\& =\EE \left[\psi_\vvbeta(X_{t_L}/\sqrt{t_L},z_1/\sqrt{t_L})(X_{t_L}+\sigma u_L)\vec{\vv p}(X_{t_L};t_L)\right]
    \\  & =\EE \left[\psi_\vvbeta(X_{t_L}/\sqrt{t_L},z_1/\sqrt{t_L})(X_{t_L}+\sigma u_L)
\vec{\vv p}(X_{t_{L-1}};t_{L-1})\right],
 \end{align*}
 where in the last step we used the fact that $t_{L-1}>t_L$, and that if $\{M_t\}_{t\ge 0}$ is a martingale with respect to the filtration $\{\calF_t\}_{t\ge0}$, then for all $Y$ measurable with respect to $\calF_s$, and $0\le s<t$, $\EE (YM_s) = \EE (YM_t)$ provided that integrability is guaranteed.

 Hence
 \begin{align*}
    (\bfA+t_{L-1}\bfC & +u_{L-1}\bfB)(\bfA+t_L\bfC+u_L\bfB)|V_\vvbeta(z_1)\rangle
   \\ &= \EE \left[\psi_\vvbeta(X_{t_L}/\sqrt{t_L},z_1/\sqrt{t_L})(X_{t_L}+\sigma u_L)
  (\bfA+t_{L-1}\bfC+u_{L-1} \bfB)\vec{\vv p}(X_{t_{L-1}};t_{L-1})\right]
    \\& = \EE\left[\psi_\vvbeta(X_{t_L}/\sqrt{t_L},z_1/\sqrt{t_L})(X_{t_L}+\sigma u_L)(X_{t_{L-1}} +\sigma u_{L-1})\vec{\vv p}(X_{t_{L-1}};t_{L-1}) \right]
   \\ & =
   \EE\left[\psi_\vvbeta(X_{t_L}/\sqrt{t_L},z_1/\sqrt{t_L})(X_{t_L}+\sigma u_L)
   (X_{t_{L-1}} +\sigma u_{L-1})
   \vec{\vv p} (X_{t_{L-2}};t_{L-2})\right],
 \end{align*}
 where in the last step we used the martingale property and $t_{L-2}>t_{L-1}$.
 Proceeding recursively, we get %
 \begin{align*}
   \prod_{k=1}^L (\bfA+t_k\bfC+u_k\bfB)|V_\vvbeta(z_1)\rangle   & =
 \EE\left[\prod_{k=1}^L(X_{t_k}+\sigma u_k)  \vec{\vv p}(X_{t_{1}};t_{1})\psi_\vvbeta(X_{t_L}/\sqrt{t_L},z_1/\sqrt{t_L}) \right].
 \end{align*}
 Since
 \[
 \langle W_\vvalpha(z_0)|  \vec{\vv p}(X_{t_1};t_1 )=
 \sum_{n=0}^\infty \alpha_n z_0^n (t_1 )^{n/2} p_n(X_{t_1}/\sqrt{t_1})=
  \phi_\vvalpha(X_{t_1}/\sqrt{t_1},z_0 \sqrt{t_1 }),
  \]
  this
  ends the proof of \eqref{Mkv-ans0}.
\arxiv
{Indeed, let  $\|f(\cdot)\|_\infty=\sup_{x\in \mathrm{supp} (\nu)}|f(x)|$. For   $N=1,2,\dots$, we have
\begin{multline*}
\left| \langle W_\vvalpha(z_0)|  \EE\left[\prod_{k=1}^L(X_{t_k}+\sigma u_k)  \vec{\vv p}(X_{t_{1}};t_{1})\psi_\vvbeta(X_{t_L}/\sqrt{t_L},z_1/\sqrt{t_L}) \right] - {\EE\left[ \phi_\vvalpha\pp{\frac{X_{t_1}}{\sqrt{t_1}},z_0\sqrt{t_1 }}\psi_\vvbeta\pp{\frac{X_{t_L}}{\sqrt{t_L}},\frac{z_1}{\sqrt{t_L}}} \prod_{k=1}^L(\sigma u_k+X_{t_k}) \right]}\right|
\\ \leq
\Big|\sum_{n=0}^{N-1} \alpha_n z_0^n (t_1 )^{n/2} \EE\left[\prod_{k=1}^L(X_{t_k}+\sigma u_k)  p_n(X_{t_1}/\sqrt{t_1})\psi_\vvbeta(X_{t_L}/\sqrt{t_L},z_1/\sqrt{t_L}) \right] \hfill
\\
-\EE\left[\sum_{n=0}^{N-1} \alpha_n z_0^n (t_1 )^{n/2} \prod_{k=1}^L(X_{t_k}+\sigma u_k)  p_n(X_{t_1}/\sqrt{t_1})\psi_\vvbeta(X_{t_L}/\sqrt{t_L},z_1/\sqrt{t_L}) \right] \Big|
\\
+\Big|\sum_{n=N}^{\infty} \alpha_n z_0^n (t_1 )^{n/2} \EE\left[\prod_{k=1}^L(X_{t_k}+\sigma u_k)  p_n(X_{t_1}/\sqrt{t_1})\psi_\vvbeta(X_{t_L}/\sqrt{t_L},z_1/\sqrt{t_L}) \right]\Big|
\\
+ \Big|\EE\left[\sum_{n=N}^{\infty} \alpha_n z_0^n (t_1 )^{n/2} \prod_{k=1}^L(X_{t_k}+\sigma u_k)  p_n(X_{t_1}/\sqrt{t_1})\psi_\vvbeta(X_{t_L}/\sqrt{t_L},z_1/\sqrt{t_L}) \right]\Big|
\\
\leq 2  MC \sum_{n=N}^\infty r_0^n  \alpha_n \|p_n(\cdot )\|_\infty \to 0 \mbox{ as } N\to\infty.
\end{multline*}
where $r_0=|z_0|\sqrt{t_1}<%
z_\vvalpha$,
$$
M=\prod_{j=1}^L \sup_{x\in\mathrm{supp}(\nu)}|x \sqrt{t_j}+\sigma u_j |
\mand C=\|\psi_\vvbeta(\cdot,z_1/\sqrt{t_L})\|_\infty.$$
}
\end{proof}

\subsection{Formulae with constant step weights and geometric boundary weights} %

We have shown in Proposition \ref{prop:markov} how to represent the probability generating function in terms of expectations of certain Markov processes. To make use of such a representation, we would like to work with Markov processes with explicit formulae, and also the appropriate choice of boundary weights $\vv\alpha,\vvbeta$ so that the introduced functions $\phi_\vvalpha,\psi_\vvbeta$ have simple formulae.

From now on we restrict to constant step weights and geometric boundary weights. For convenience we recall them here:
\equh\label{eq:abc}
\vv a=(1,1,\dots), \quad \vv b=(\sigma,\sigma,\dots), \quad \vv c=(1,1,\dots), \quad \alpha_n = \rho_0^n, \quad \beta_n = \rho_1^n,
\eque
with $\sigma>0, \rho_0,\rho_1>0, \rho_0\rho_1<1$.
The corresponding orthogonal polynomials
 (depending on $\vv a,\vv b, \vv c$ alone) are determined by
\[
x p_n(x)=p_{n+1}(x)+p_{n-1}(x).
\]
 Thus, we are now dealing directly with the Chebyshev polynomials of the second kind. It is well known that the associated measure is the semi-circular law
\[
\nu(dx) = \frac{\sqrt{4-x^2}}{2\pi}\inddd{|x|\le 2}dx.
\]
It is also well known that $|p_n(x)|\leq n+1$ on the support $[-2,2]$ of $\nu$, that $\|p_n\|_2^2=1$,
and that the generating function is
\[
\phi(x,z):=\sum_{n=0}^\infty z^n p_n(x)=\frac{1}{1-xz+z^2}, \; |z|<1, x\in[-2,2].
\]
\arxiv{Since $|p_n(x)|\leq p_n(2)=n+1$, the series converges uniformly in $x\in[-2,2]$, provided $|z|<1$.}
(The above formulas follow from \cite[(4.5.28), (4.5.20), (4.5.23)]{ismail09classical} by a change of variable $x$ to $x/2$.)

The Markov processes and orthogonal martingale polynomials in Assumption \ref{assumption:1} in this case have been studied.
It is known \citep[Remark 4.1]{Bryc-Wesolowski-03} that  the functions $\{p_n(x;t)\}_{t> 0}$ defined by \eqref{pt} are then orthogonal martingale  polynomials for a  Markov process $(X_t)_{t\geq 0}$  with  $X_0=0$ and with univariate distributions $P(X_t\in dx)=p_t(x)dx$ given by
\[%
   p_t(x)= \frac{\sqrt{4t-x^2}}{2\pi t}1_{\ccbb{|x|\le 2\sqrt t}}
  ,\quad t>0,
\]
or $p_t = \nu_t$ with $\nu_t$ determined by a dilation of $\nu$, and %
   with transition probabilities $P(X_t\in dy\mid X_s=x)=p_{s,t}(x,y)dy$ for $0\leq s<t$   given by
\[%
  p_{s,t}(x,y)=\frac{1}{2\pi} \frac{(t-s)\sqrt{4t-y^2}}{tx^2+sy^2-(s+t)xy+(t-s)^2}\; \mbox{ for $|x|\leq 2\sqrt{s}, |y|\leq 2\sqrt{t}$}.
\]%
With geometric weights \eqref{eq:geometric}, the functions $\phi_\vvalpha,\psi_\vvbeta$ now can be expressed as
\[%
\phi_\vvalpha(x,z)=\phi(x,z\rho_0) \qmand \psi_\vvbeta(x,z)=\phi(x,z\rho_1),
\]%
for $z$ such that $|z|<z_\vv\alpha:=1/\rho_0$ and $|z|<z_\vvbeta:=1/\rho_1$, respectively, and the series converge uniformly in $x\in[-2,2]$ .

Combining the above with Proposition \ref{prop:markov}, we have arrived at the following. Note that $\Pr_L$, the probability measure on $\calM\topp L$, depends now on $\sigma,\rho_0,\rho_1$. We let $\EE$ also denote the expectation for functionals of the associated Markov process $\{X_t\}_{t\ge 0}$. %
\begin{proposition} Assume \eqref{eq:abc}. If $\rho_0\rho_1<1$, $t_1\geq t_2\geq \dots\geq t_L>0$, $u_1,\dots,u_L>0$,  and $z_0,z_1$ are close enough to 0 so that %
\equh\label{eq:cond}
\rho_0 |z_0|\sqrt{t_1}<1 \qmand \frac{\rho_1 |z_1|}{\sqrt{t_L}}< 1,
\eque
 then
 \begin{equation}\label{eq:Markov}
 \EE_L \left[z_0^{\gamma_0}\prod_{k=1}^L t_k^{\eps_k^-}u_k^{\eps_k^0} z_1^{\gamma_L}
    \right] =
     \frac{1}{\mathfrak C_L}
     \EE\left[  \frac{\prod_{k=1}^L(\sigma u_k+X_{t_k})  }
     {(1-\rho_0 z_0 X_{t_1}+\rho_0^2 z_0^2t_1 )(1-\rho_1 z_1 X_{t_L  }/{t_L  }+\rho_1^2z_1^2/t_L  )}\right].
\end{equation}
Here $\mathfrak C_L$ is the normalization constant \eqref{CL}.
\end{proposition}

\begin{remark}
  If one is interested only in  Theorem \ref{thm:1}(i), then \eqref{eq:Markov} can be simplified as  follows. With $z_0$ and $\tau_1<\tau_2<\dots<\tau_L$ small enough   we have %
\begin{equation}
  \label{q=0-GenFun-EZ}
\EE_L\left[z_0^{\gamma_0}\prod_{k=1}^L \tau_k^{\gamma_k-\gamma_{k-1}}
    \right] =
     \frac{1}{\mathfrak C_L}  \EE\left[  \frac{\prod_{k=1}^L(\sigma +\tau_k X_{1/\tau_k^2})  }
     {(1-\rho_0 z_0 X_{1/\tau_1^2 }+\rho_0^2 z_0^2/\tau_1^2 )(1-\rho_1 \tau_L^2 X_{1/\tau_L^2 }  +\rho_1^2\tau_L^2  )}\right].
\end{equation}
To see this, we use \eqref{eq:Markov} with $z_0=z_0$, $z_1=1$; we take $t_k=1/\tau_k^{2}$ and $u_k=1/\tau_k$. After  multiplying both sides  by $\tau_1\dots \tau_L$,
on the left-hand side of \eqref{q=0-GenFun-EZ} we get
$ %
\EE_L
\left[z_0^{\gamma_0}\prod_{k=1}^L \tau_k^{1-2\eps_k^- -\eps_k^0}
    \right] $. To complete the derivation we note that $1-2\eps_k^- -\eps_k^0=(\eps_k^++\eps_k^0+\eps_k^-)-2\eps_k^- -\eps_k^0=
\eps_k^+-\eps_k^-=\gamma_k-\gamma_{k-1}$.
\end{remark}

The integral formula for
the normalization constant  $\mathfrak C_L$, however, will require additional effort
as we want to include the case where
$\rho_1$
can be larger than 1 in our asymptotic analysis.
In particular the following representation of $\mathfrak C_L$ will be useful.
\begin{proposition} Assume $\rho_0\in(0,1), \rho_0\rho_1\in(0,1)$. Then,
\equh\label{ZL-MP}
\mathfrak C_L%
=\int_\RR\frac{ (x+\sigma)^L}{1-x \rho_0 +\rho_0^2 }  \mu_{\rho_1}(dx),
\eque
where the probability  measure $\mu_{\rho_1}$ of a possibly mixed type is given by
  \begin{equation}\label{M-P}
    \mu_\rho(dx)=\frac{1}{2\pi} \frac{\sqrt{4-x^2}}{1-x \rho +\rho^2 }\inddd{|x|<2}dx+\left(1-\frac{1}{\rho^2}\right)_+ \delta_{\rho+\frac1\rho}(dx).
  \end{equation}
 (Here $x_+:=\max\{0,x\}$.)
\end{proposition}
We remark that measure \eqref{M-P} is a shifted Marchenko--Pastur law. %
\begin{proof}
We first note that the result holds if both $\rho_0, \rho_1<1$.  Indeed, in this case, we can apply \eqref{eq:Markov} with $z_0=z_1=t_k=u_k=1$. Then the left-hand side of  \eqref{eq:Markov} is $1$,  so the right-hand side gives  the integral formula for $\mathfrak C_L$ that we want.
We now fix $\rho_0\in(0,1)$. As a function of $\rho_1$, this explicit integral formula extends analytically to complex argument, defining a function
\equh\label{f1}
   f(\rho)  =\frac{1}{2\pi}\int_{-2}^{2}\frac{ (\sigma+x)^L}{(1-x \rho_0 +\rho_0^2)(1-x \rho +\rho^2)}\sqrt{4-x^2}d x,
\eque
which is analytic  in the complex unit disk  $|\rho|<1$.

Next we note that since the edge-weights are bounded by $\max \{\sigma,1\}$, %
the function %
\begin{equation}\label{f}
  \mathfrak C(\rho)=\sum_{i,j=0}^\infty \rho_0^i \rho^j \mathfrak W_{i,j}\topp L,
\end{equation}
is analytic in the complex disk $|\rho|<1/\rho_0$ (see \eqref{CL}).
Since we deduced from \eqref{eq:Markov} that $f(\rho)=\mathfrak C(\rho)$ for $\rho\in(0,1)$,  expression \eqref{f} coincides with \eqref{f1} for $|\rho|<1$ and is its analytic extension   to the complex disk $|\rho|<1/\rho_0$.

 Our goal is to extend the integral representation \eqref{f1} to a larger domain by explicit analytic continuation.
We first re-write \eqref{f1} as a complex integral. Substituting $x=2\cos \theta$, and then $z=e^{i\theta}$, in \eqref{f1}  we get %
\begin{align}
 \frac{1}{2\pi}&\int_{-2}^{2}\frac{ (\sigma+x)^L}{(1-x \rho_0 +\rho_0^2)(1-x \rho +\rho^2)}\sqrt{4-x^2}d x\nonumber
     \\& =\frac{1}{2\pi}\int_{0}^{\pi}\frac{ 4\sin^2\theta(\sigma+2\cos \theta )^L}{(1-2 \rho_0 \cos \theta +\rho_0^2)(1-2 \rho\cos \theta+\rho^2)}  d\theta
    \nonumber\\
&   =\frac{1}{4\pi}\int_{-\pi}^{\pi}\frac{ 4\sin^2\theta(\sigma+2\cos \theta )^L}{(1-2 \rho_0 \cos \theta +\rho_0^2)(1-2 \rho\cos \theta+\rho^2)}  d\theta\nonumber
  \\ &=
   \pp{-\frac{1}{2}}\cdot \frac{1}{2\pi i}\oint_{|z|=1}\frac{(z^2-1)^2\left(\sigma+z+\frac1z\right)^L}{
   (1-\rho_0 z)(z-\rho_0)(1-\rho z)(z-\rho)} \frac{dz}{z},\label{f2}
\end{align}
which is valid for all $|\rho|<1$.
Consider now $|\rho|\in(\rho_0,1)$.
In the last line above one can replace the contour $|z|=1$ by $|z|=r$, and this replacement is valid as long as the circle does not cross any pole of the integrand, that is,  for $r\in(1,1/|\rho|)$.
Next, let the contour cross the pole at $1/\rho$ and (because of the additional factor $-1/2$ in front of the integral)
 add half of the residue at $z=1/\rho$.
 We arrive at %
\equh\label{f3}
f_1(\rho)=   -\frac{1}{4\pi i}\oint_{|z|=r}\frac{(z^2-1)^2\left(\sigma+z+\frac1z\right)^L}{
   (1-\rho_0 z)(z-\rho_0)(1-\rho z)(z-\rho)} \frac{dz}z +\frac12 \frac{\left(\rho ^2-1\right) \left(\rho
   +\frac{1}{\rho }+\sigma \right){}^L}{
   \rho  \left(\rho -\rho _0\right)
   \left(1-\rho _0 \rho \right)},
   \eque
which %
coincides with $f(\rho)$ for $|\rho|\in(\rho_0,1)$, %
by deformation of contour as explained above,
and, with $r$ fixed, gives the analytic extension of $f$ to  all $\rho$ such that $1/r<|\rho|<r$.
In particular, $\mathfrak C(\rho)=f_1(\rho)$  for $1/r<|\rho|<r$. Note that $r$ can be taken arbitrarily close to $1/\rho_0$.

   Next, consider the expression \eqref{f3} for $r\in(1,1/\rho_0)$ and $\rho$ such that $|\rho|\in(1,r)$, and deform the contour of integration back to $|z|=1$. This
subtracts
    half of the residue of the integrand at $z=\rho$.
   Since $r$ can be taken arbitrarily close to $1/\rho_0$,
 $f_1(\rho)$  is equal to
\equh\label{f4}
   f_2(\rho)=
-    \frac{1}{4\pi i}\oint_{|z|=1}\frac{(z^2-1)^2\left(\sigma+z+\frac1z\right)^L}{
   (1-\rho_0 z)(z-\rho_0)(1-\rho z)(z-\rho)} \frac{dz}z
   +\frac{\left(\rho ^2-1\right) \left(\rho
 +\frac{1}{\rho }+\sigma \right){}^L}{\rho
  \left(\rho -\rho _0\right) \left(1-\rho _0
   \rho \right)}
   \eque
   for all $\rho$  such that  $|\rho|\in(1,1/\rho_0)$.
In particular, $\mathfrak C(\rho)=f_2(\rho)$ for $|\rho|\in(1,1/\rho_0)$.
Returning back to the real arguments, we see that  \eqref{f1},  \eqref{f2} and \eqref{f4} can be combined together  into a single formula which gives
\begin{multline*}
\mathfrak C(\rho)=\frac{1}{2\pi}\int_{-2}^{2}\frac{ (\sigma+x)^L}{(1-x \rho_0 +\rho_0^2)(1-x \rho +\rho^2)}\sqrt{4-x^2}d x
\\+
 \frac{ \left(1-\frac{1}{\rho^2}\right)_+ \rho \left(\rho
 +\frac{1}{\rho }+\sigma \right)^L}{
  \left(\rho -\rho_0\right) \left(1-\rho_0
   \rho\right)}, \quad \rho\in(0,1/\rho_0),\rho\ne 1.
\end{multline*}
This formula extends to $\rho=1$ by continuity, establishing \eqref{M-P} for $\rho_1\leq 1$. To prove \eqref{M-P} for $\rho_1\in(1,1/\rho_0)$,
we note that by an elementary calculation the contribution of the atom of $\mu_{\rho_1}$ matches
the additional term arising from the residues in \eqref{f4}:
\begin{align*}
        \int_\RR \frac{(x+\sigma)^L}{1-x\rho_0+\rho_0^2}\pp{1-\frac1{\rho^2}}\delta_{\rho+1/\rho}(dx)
 &       =\pp{1-\frac1{\rho^2}}
           \frac{ (\sigma+x)^L}{1-x \rho_0 +\rho_0^2}\Big|_{x=\rho+1/\rho}
  \\& =\pp{1-\frac1{\rho^2}}\frac{\rho  \left(\rho +\frac{1}{\rho }+\sigma
   \right)^L}{\left(\rho -\rho_0\right)
   \left(1-\rho  \rho _{0 }\right)}.
\end{align*}
\end{proof}

\section{Proof of Theorem \ref{thm:1}}\label{sec:proof}
We first note that since $U_L(x)+H_L(x)+D_L(x)=\floor{Lx}$, it is enough to prove joint convergence of two of the three processes. We will show that
\begin{equation}
   \label{path2}
   \frac{1}{\sqrt{L}}\ccbb{\pp{\gamma_{\floor{Lx}}, H_L(x)-\frac{\sigma }{2+\sigma}\floor{Lx}}}_{x\in[0,1]}\fddto\ccbb{\left(\frac{\sqrt{2}}{\sqrt{2+\sigma}} \wt \eta_x\topp {\A',\C'}, \frac{\sqrt{2\sigma}}{2+\sigma} B_x\right)}_{x\in[0,1]},
\end{equation}
where $\wt \eta\topp {\A',\C'}$, $B$ are independent and %
$\A',\C'$ are given by \eqref{a'c'}.
The above implies the desired joint convergence of $U_L(x), H_L(x), D_L(x)$, as
\begin{align*}
U_L(x)+D_L(x) & = \floor{Lx} - H_L(x),\\
U_L(x)-D_L(x) & = \gamma_{\floor{Lx}} - \gamma_0.
\end{align*}
\arxiv{
Indeed, since $U_L(x)-D_L(x)=\gamma_{\floor{Lx}}-\gamma_{0}$ and $H_L(0)=0$, convergence in \eqref{path2} implies
 $$\frac{1}{\sqrt{L}}\ccbb{\pp{U_L(x)-D_L(x), H_L(x)-\frac{\sigma Lx}{2+\sigma}}}_{x\in[0,1]}\fddto\ccbb{\left(\frac{\sqrt{2}}{\sqrt{2+\sigma}}   \eta_x\topp {\A',\C'}, \frac{\sqrt{2\sigma}}{2+\sigma} B_x\right)}_{x\in[0,1]}.$$
Since
$U=\frac{1}{2}(U-D)+\frac12(U+D)=\frac12(U-D)+\frac12(\floor{Lx}-H_L(x))
=\frac12(U_L-D_L)+\frac12(\frac{ \sigma \floor{Lx}}{2+\sigma}-H_L(x))+\frac{\floor{Lx}}{2+\sigma} $,
we see that the first component in \eqref{UHD} is %
$$U_L(x)-\frac{\floor{Lx}}{2+\sigma}=\frac{1}{2}\pp{
U_L(x)-D_L(x)}-\frac{1}{2}\pp{ H_L(x)-\frac{\sigma \floor{L x}}{2+\sigma}}=\frac12(\gamma_{\floor{Lx}}-\gamma_0)-\frac{1}{2}\pp{ H_L(x)-\frac{\sigma \floor{L x}}{2+\sigma}}$$
is a linear combination of the pair of processes in \eqref{path2}.

Similarly, $D=\frac12(U+D)-\frac{1}{2}(U-D)$, so
$$D_L(x)-\frac{\floor{Lx}}{2+\sigma}=-\frac12(\gamma_{\floor{Lx}}-\gamma_0)-\frac{1}{2}\pp{ H_L(x)-\frac{\sigma \floor{L x}}{2+\sigma}}.
$$
(And of course $(B_x)$ has the same law as $(-B_x)$.)
}

To prove \eqref{path2}, we fix  $d\in\N$, and $\vv x=(x_1,\dots,x_d)$
 with
$x_0:=0< x_1<\cdots<x_d=1$. Denoting %
\[
h_k=H_L(x_k)-\frac{\sigma }{2+\sigma}\floor{L x_k},
\]
 we introduce the Laplace transform $\Phi_{L}$ by the formula
\begin{equation}
  \label{bi-Laplace-0}
   \Phi_L (\vv c, \vv\theta):= \EE_L\bb{\exp\pp{-\sum_{k=0}^d c_k \gamma_{\floor{Lx_k}}+\sum_{k=1}^d \theta_kh_{k}}}.
\end{equation}
In this section, recall that $\EE_L$ is
the  expected value
with respect to
probability measure
$\Pr_L$ on $\calM\topp L$
defined by formula \eqref{Pr0} with parameters \eqref{rho(L)}.
Since $\vv x$ is fixed throughout this proof, we suppress dependence of $\Phi_L$ on $\vv x=(x_0,\dots,x_d)$ in our notation.

Our goal is to compute $\Phi_L (L^{-1/2}\vv c, L^{-1/2}\vv\theta)$ and identify the limit.
The main step in the proof is to show that the expression for the limiting Laplace transform factors and takes the following form.
\begin{proposition}\label{L-limL} Fix $d=1,2,\dots$.
If $\theta_1, \theta_2, \dots, \theta_d\in\RR$    and $c_0,c_1,\dots,c_d>0$
are such that
\equh\label{eq:C}
\C+c_0>0 \qmand  c_d+\A>0,
\eque
 then
  \begin{equation}\label{Lim-ph}
    \lim_{L\to\infty}
    \Phi_L \pp{\frac{\vv c}{\sqrt{L}}, \frac{\vv\theta}{\sqrt{L}}}=  \Psi (\vv c)
      \cdot \exp\pp{\frac{\sigma}{(2+\sigma)^2}\sum_{k=1}^d  (x_k-x_{k-1})\wt s_k^2},
  \end{equation}
with %
$\wt s_k=\sum_{j=k}^d \theta_j, k=1,\dots,d$, and

  \begin{equation}\label{lim-ps}
    \Psi (\vv c)=\frac{\sqrt {2+\sigma}}{\sqrt{2}\pi\mathfrak C_{\A',\C'}} \int_{\RR_+^{d}} e^{-\frac{1}{2+\sigma}\sum_{k=1}^d (x_k-x_{k-1})u_k}
 f(u_1)g(u_d)
  \prod_{k=1}^{d-1} \p_{c_k}(u_{k},u_{k+1})\, d\vv u,
  \end{equation}
  with
\[%
    f(u_1)=\frac{\sqrt{u_1}}{(\C+c_0)^2+u_1}, \quad
     g(u_d)=\frac{1}{(\A+c_d)^2+u_d},
\]
  and
  \begin{equation}\label{biane-p}
  \p_t(x,y) %
  =\frac{2}{\pi}\cdot \frac{t \sqrt{y}}{t^4+(y-x)^2+2 (y+x) t^2}.
\end{equation}

\end{proposition}
The proof of this proposition is postponed to Section \ref{sec:L-limL}.
The transition probability density \eqref{biane-p} appeared as the tangent process in \cite{Bryc-Wang-2015}, and later
as the square of the  radial  part  of  a  3-dimensional  Cauchy  process in \cite[Corollary 1]{KyprianouOConnell2021}.

   The second step in the proof of Theorem \ref{thm:1} is to
  use the following re-write of  \cite[Proposition 4.10]{Bryc-Wang-Wesolowski-2022} which uses a change of variables and the self-similarity $\p_{at}(a^2 x,a^2 y)=\p_t(x,y)/a^2$,
  $\g_{a^2 x}(a z, a z')=\g_{x}(z,z')/a$, $a>0$, of the kernels \eqref{biane-p} and \eqref{gg} to insert an auxiliary parameter $\tau>0$ into the formula.

\begin{proposition}\label{prop:duality} Let $f,g$ be two measurable functions on $\RR_+$ and $d= 2,3,\dots$.
With   $\tau,c_1,\cdots,c_{d-1}>0$,   and $0= x_0<x_1<\cdots<x_d\leq 1$, we have

\begin{multline}\label{eq:duality}
\int_{\R_+^{d}}e^{-\tau \summ k1d(x_k-x_{k-1})u_k}f(u_1)\left(\prodd k1{d-1}\p_{c_k}(u_{k},u_{k+1})\right)g(u_{d}) %
d\vv u
\\
=  \frac4\pi\int_{\R_+^{d-1}}e^{- \sum\limits_{k=1}^{d-1}c_k z_k}\what f(z_1)\left(\prodd k2{d-1}\g_{2\tau(x_k-x_{k-1})}(z_{k-1},z_k)\right)\what g(z_{d-1})d\vv z,
\end{multline}
where %
\begin{align}
\label{hat:f}\what f(z) & := \int_{\R_+}f(u^2)\sin(uz)e^{-\tau x_1u^2}du,\\
\label{hat:g}\what g(z)&:= \int_{\R_+}g(u^2)u\sin(uz)e^{-\tau (x_d-x_{d-1})u^2}du,
\end{align}
provided that  the functions under the multiple integrals  in \eqref{eq:duality} are absolutely integrable.
\end{proposition}
  \arxiv{
  \begin{proof}
 The following is essentially the proof of \cite[Proposition 4.10]{Bryc-Wang-Wesolowski-2022}, with minor modifications:

 With a change of variables $u_k\mapsto u_k^2$, the left-hand side of \eqref{eq:duality} becomes
\begin{align}
2^{d}\int_{\R_+^{d}}e^{{-\tau\summ k1d(x_k-x_{k-1})u_k^2}}u_1f(u_1^2)\left(\prodd k1{d-1}u_{k+1}\p_{c_k}(u_{k}^2,u_{k+1}^2)\right)g(u_{d}^2)d\vv u.\label{eq:duality1}
\end{align}
Recall from \cite[after Eq.~(2.2)]{bryc18dual} that
\begin{equation}\label{eq:ypc}
y\p_t(x^2,y^2) = \frac1\pi \frac yx\int_{\R_+} e^{-tz}\sin(xz)\sin(yz)dz,
\end{equation}
and
\begin{equation}\label{trig2g}
  \frac{2}{\pi}\int_{\R_+} e^{-tx^2/2}\sin(xy_1)\sin (xy_2)dx =  \g_t(y_1,y_2).
\end{equation}
Then,
\eqref{eq:duality1} becomes
 \begin{align*}
&  \frac{2^{d}}{\pi^{d-1}}  \int_{\R_+^{d}}u_1f(u_1^2) \pp{\prodd k1{d-1}\frac{u_{k+1}}{u_{k}}\int_{\R_+}e^{-c_kz_k}\sin(u_kz_k)\sin(u_{k+1}z_k)dz_k} g(u_{d}^2)\; e^{-\tau\sum\limits_{k=1}^d(x_k-x_{k-1})u_k^2}d\vv u\\
&  = \frac 4\pi\int_{\R_+^{d-1}}e^{-\sum\limits_{k=1}^{d-1}c_kz_k}\Bigg(\int_{\R_+}e^{-\tau x_1 u_1^2} f(u_1^2)\sin(u_1z_1)du_1  \prodd k2{d-1}\frac2\pi\int_{\R_+}e^{-\tau(x_k-x_{k-1})u_k^2}\sin(u_kz_k)\sin(u_kz_{k-1})du_k \\
&  \quad\times \int_{\R_+}u_{d} e^{-\tau(x_d-x_{d-1})u_d^2} \sin(u_{d}z_{d-1})g(u_{d}^2)du_d \Bigg)d\vv z\\
&=
\frac 4\pi\int_{\R_+^{d-1}}e^{-\sum\limits_{k=1}^{d-1}c_kz_k}\Bigg(\int_{\R_+}e^{-\tau x_1 u_1^2} f(u_1^2)\sin(u_1z_1)du_1  \prodd k2{d-1}\g_{2\tau(x_k-x_{k-1})}(z_{k-1},z_k)   \int_{\R_+}u_{d} e^{-\tau(x_d-x_{d-1})u_d^2} \sin(u_{d}z_{d-1})g(u_{d}^2)du_d \Bigg)d\vv z
\\ &=
\frac 4\pi\int_{\R_+^{d-1}}e^{-\sum\limits_{k=1}^{d-1}c_kz_k}\Bigg(\what f(z_1) \prodd k2{d-1}\g_{2\tau(x_k-x_{k-1})}(z_{k-1},z_k)  \what g(z_{d-1})\Bigg)d\vv z
.
\end{align*}
\end{proof}
}

 We apply Proposition \ref{prop:duality} to \eqref{lim-ps},  using an auxiliary Markov process $\zeta$ with  transition probabilities $\Pr(\zeta_t\in dy|\zeta_s=x)=\p_{t-s}(x,y)dy$ for $s<t$  with density
 \eqref{biane-p}.

\begin{proposition}\label{prop:dual}  For $k=1,\dots, d$, let $s_k=\sum_{j=k}^dc_j$, where $d=1,2,\dots$. Then
 \begin{multline}
 \label{zeta:dual} \frac{\sqrt{2+\sigma}}{\sqrt 2}\frac{1}{\pi \mathfrak C_{\A',\C'}} \int_0^\infty \EE\bb{ \frac{e^{-\frac{1}{2+\sigma}\sum_{k=1}^d (x_k-x_{k-1})\zeta_{s_k}}}{\zeta_{s_1}+(\C+c_0)^2}\middle | \zeta_{s_d}=u} \frac{\sqrt{u}\, du}{(\A+c_d)^2+u}
 \\
 = \EE\bb{e^{ -\frac{\sqrt{2}}{\sqrt{2+\sigma}}\sum_{k=0}^dc_k\wt \eta_{x_k}\topp{\A',\C'}}}.
\end{multline}
\end{proposition}
\begin{proof}
 We use Proposition \ref{prop:duality} with  $\tau=1/(2+\sigma)$, $f(u)=\sqrt{u}/((\C+c_0)^2+u)$ and
  $g(u)=1/((\A+c_d)^2+u)$.
 Since  $\int_0^\infty e^{-s z}\sin(u z)dz=\frac{u}{s^2+u^2}$,
we have
\begin{align*}f(u^2)
& =\frac{u}{(\C+c_0)^2+u^2}=\int_0^\infty e^{-(\C+c_0)z_0}\sin (z_0 u)dz_0,\\
u g(u^2) & =\frac{u}{(\A+c_d)^2+u^2}= \int_0^\infty e^{-(\A+c_d)z_d}\sin (z_d u)dz_d.
\end{align*}
\arxiv{Thus, if $d=1$, after change  of variable $u\mapsto u^2$ and application of \eqref{trig2g}, the left-hand side of
\eqref{zeta:dual} becomes
\begin{multline*}
 \frac{1}{\sqrt{2\tau }\pi\mathfrak C_{\A',\C'}} \int_0^\infty e^{- \tau x_1  u}
 f(u)g(u) d u
 = \frac{2}{\sqrt{2\tau }\pi\mathfrak C_{\A',\C'}}\int_0^\infty e^{-\tau x_1 u^2}
 \frac{ u^2}{((\C+c_0)^2+u^2)((\A+c_1)^2+u^2)}  d u
 \\
 = \frac{1}{\sqrt{2\tau }\mathfrak C_{\A',\C'}}
 \int_{\RR_+^2} e^{-(\C+c_0)z_0} e^{-(\A+c_1)z_1}
 \Big(\frac{2}{\pi}\int_0^\infty e^{- \tau x_1  u^2}  \sin (z_0 u)\sin (z_1 u)  d u\Big) dz_0   dz_1
 \\ =\frac{1}{\sqrt{2\tau}\mathfrak C_{\A',\C'}}
 \int_{\RR_+^2} e^{-(\C+c_0)z_0} e^{-(\A+c_1)z_1}\g_{2\tau x_1}(z_0,z_1)
   dz_0   dz_1
\end{multline*}
Since  $\g_{2\tau x}(z_0,z_1)=\frac{1}{\sqrt{2\tau}}\g_x(z_0/\sqrt{2\tau},z_1/\sqrt{2\tau})$, changing the variables to $y_j=z_j/\sqrt{2\tau}$, $j=0,1$,  we see that the left hand side of \eqref{zeta:dual}  is
$$\int_{\RR_+^2} e^{-(\sqrt{2\tau} c_0 y_0+\sqrt{2\tau} c_1 y_1)} \wt{ p}_{0,x_1}\topp{\A',\C'}(y_0,y_1)d y_0 d y_1
= \EE \left[e^{-\sqrt{2\tau} (c_0\wt \eta_0\topp{\A',\C'}+c_1 \wt \eta_{x_1}\topp{\A',\C'})}\right],$$
see \eqref{eq:eta_pdf} and \eqref{a'c'}. This proves \eqref{zeta:dual} for $d=1$.
}
Next, consider $d\geq 2$.
Formulas \eqref{hat:f} and \eqref{hat:g} become
\begin{align*}
 \what f(z_1) & = \int_{\R_+}f(u^2)\sin(uz_1)e^{-\tau x_1u^2}du =
\int_0^\infty e^{-(\C+c_0)z_0}\int_{\R_+}e^{-\tau x_1u^2}\sin(uz_1)\sin (z_0 u)du dz_0
\\ &=\frac{\pi}{2}\int_{\RR_+}  e^{-(\C+c_0)z_0} \g_{2\tau x_1}(z_0,z_1)dz_0,
\end{align*}
and
\begin{align*}
   \what g(z_{d-1}) & =\int_{\R_+}ug(u^2)\sin(uz_{d-1})e^{-\tau (x_d-x_{d-1})u^2}du
\\
& =\int_{\R_+}\pp{\int_0^\infty e^{-(\A+c_d)z_d}\sin (z_d u)dz_d}\sin(uz_{d-1})e^{-\tau (x_d-x_{d-1})u^2}du
\\& =
\int_{\RR_+}e^{-(\A+c_d)z_d}\pp{\int_{\RR_+} e^{-\tau (x_d-x_{d-1})u^2}\sin (z_d u)\sin(uz_{d-1})du} dz_d
\\& = \frac{\pi}{2} \int_{\R_+}  e^{-(\A+c_d)z_d} \g_{2\tau(x_d-x_{d-1})}(z_{d-1},z_d) dz_d.
\end{align*}
Thus by \eqref{eq:duality}, the left-hand side of
\eqref{zeta:dual} becomes
\begin{align*}
 \frac{\sqrt{2+\sigma}}{\sqrt 2}&\frac{1}{ \mathfrak C_{\A',\C'}}  \int_{\R_+^{d+1}}e^{- \sum\limits_{k=0}^{d}c_k z_k} e^{-\C z_0} \left(\prodd k1{d}\g_{2\tau(x_k-x_{k-1})}(z_{k-1},z_k)\right) e^{-\A z_d} d\vv z
  \\& =  \frac{\sqrt{2+\sigma}}{\sqrt 2}\frac{1}{ \mathfrak C_{\A',\C'}} \int_{\R_+^{d+1}}e^{- \sum\limits_{k=0}^{d}c_k z_k} e^{-\C z_0} \left(\prodd k1{d}\g_{2\tau(x_k-x_{k-1})}(z_{k-1},z_k)\right) e^{-\A z_d} d\vv z\nonumber
  \\ & =\frac{\sqrt{2+\sigma}}{\sqrt 2}
 \frac{1}{ \mathfrak C_{\A',\C'}}  (2\tau)^{-d/2}
  \int_{\R_+^{d+1}}e^{- \sum\limits_{k=0}^{d}c_k z_k} e^{-\C z_0} \left(\prodd k1{d}\g_{x_k-x_{k-1}}(z_{k-1}/\sqrt{2\tau},z_k/\sqrt{2\tau})\right) e^{-\A z_d} d\vv z,\nonumber
\end{align*}
where we used scaling $\g_{2\tau x}(z,z')=\frac{1}{\sqrt{2\tau}}\g_x(z/\sqrt{2\tau},z'/\sqrt{2\tau})$.
Substituting $z_k'=z_k/\sqrt{2 \tau}$ into the integral and dropping the primes on $z_k'$, we get
\[
 \frac{1}{ \mathfrak C_{\A',\C'}}
  \int_{\R_+^{d+1}}e^{- \sum\limits_{k=0}^{d}c_k z_k\sqrt{2\tau}} e^{-\C \sqrt{2\tau} z_0} \left(\prodd k1{d}\g_{x_k-x_{k-1}}(z_{k-1},z_k)\right) e^{-\A\sqrt{2\tau} z_d} d\vv z,
  \]
  which we recognize as the desired right-hand side of \eqref{zeta:dual}.
\end{proof}

\begin{proof}[Proof of  Theorem \ref{thm:1}]
By Proposition \ref{L-limL}, the limiting Laplace transform factors.
 Proposition \ref{prop:dual} identifies the first factor in \eqref{Lim-ph}   as the Laplace transform of the first component of the process in \eqref{path2}.  We recognize the second factor in \eqref{Lim-ph} as the Laplace transform  of the second  component  of the process in \eqref{path2}. To see this, we   write it as
\[%
  \EE\left[e^{\frac{\sqrt{2\sigma}}{2+\sigma}\sum_{k=1}^d \wt s_k (B_{x_{k}}-B_{x_{k-1}})}\right]
   =   \EE\left[e^{\frac{\sqrt{2\sigma}}{2+\sigma}\sum_{k=1}^d \theta_k B_{x_{k}}}\right].
\]%
This identifies the limit of the Laplace transforms \eqref{Lim-ph}
as a Laplace transform of a probability measure.
To conclude the proof we invoke \cite[Theorem A.1]{Bryc-Wang-2017ASEP}, which asserts that convergence of Laplace transforms on an open set
to a Laplace transform of a probability measure implies convergence in distribution.
\end{proof}

\subsection{Proof of Proposition {\protect \ref{L-limL}}} \label{sec:L-limL}
By symmetry, we assume $\C>0$.
We start by rewriting the expression \eqref{bi-Laplace-0}  solely in terms of $\eps_k^-$ and $\eps_k^0$.   The first step is to write (recall that $h_0=0$)
\begin{align*}
 \Phi_L (\vv c, \vv\theta) & =  \EE_L\bb{e^{-\gamma_0 \sum_{j=0}^d c_j}\exp\pp{-\sum_{k=1}^d (\gamma_{L_k}-\gamma_{L_{k-1}}) \sum_{j=k}^dc_j+\sum_{k=1}^d (h_{k}-h_{{k-1}})\sum_{j=k}^d \theta_j}}%
     \\ & =
  \EE_L\bb{ e^{-s_0 \gamma_0}\exp\pp{-\sum_{k=1}^ds_k \sum_{j=L_{k-1}+1}^{L_k}(\eps_j^+-\eps_j^-)+\sum_{k=1}^d \wt s_k \sum_{j=L_{k-1}+1}^{L_k}\pp{ \eps_j^0-\frac\sigma{2+\sigma}}}},%
\end{align*}
with
\[%
 L_k=\floor{L x_k},\quad  s_k=\sum_{j=k}^d c_j,\; k=0,\dots,d,\quad \wt s_k=\sum_{j=k}^d \theta_j,\quad k=1,\dots,d.
\]%
Since $\eps_k^+-\eps_k^-=1-\eps_k^0-2 \eps_k^-$, we get
 \begin{align*}
   \Phi_L (\vv c, \vv\theta) & =e^{-\sum_{k=1}^d (L_k-L_{k-1})(s_k+\sigma \wt s_k /(2+\sigma))}\\
& \quad   \times  \EE_L\bb{ e^{-s_0 \gamma_0}\exp\pp{2\sum_{k=1}^ds_k \sum_{j=L_{k-1}+1}^{L_k}\eps_j^-+\sum_{k=1}^d( \wt s_k+s_k) \sum_{j=L_{k-1}+1}^{L_k}\eps_j^0}}. \end{align*}
We therefore get
\begin{multline*}
   \Phi_L \pp{\frac{\vv c}{\sqrt L},  \frac{\vv\theta}{\sqrt L}}
   \\=\prod_{k=1}^d t_{L,k}^{-(L_k-L_{k-1})/2} \prod_{k=1}^d v_{L,k}^{- (L_k-L_{k-1})\frac{\sigma}{2+\sigma}} \EE_L \left[z_{L,0}^{\gamma_0}\pp{\prod_{k=1}^d t_{L,k}^{\sum_{j=L_{k-1}+1}^{L_k}\eps_j^-}u_{L,k}^{\sum_{j=L_{k-1}+1}^{L_k}\eps_j^0}}
    \right],
\end{multline*}
  with
\[%
  z_{L,0}=e^{-s_0/\sqrt{L}},\quad   t_{L,k}=e^{2 s_k/\sqrt{L}}, \quad%
    v_{L,k}=e^{ \wt s_k/\sqrt{L}}, \quad%
   u_{L,k}=\sqrt{t_{L,k}}v_{L,k}. %
\]%

Next, we apply the Markov representation \eqref{eq:Markov},  but before
that
we verify that
\eqref{eq:cond} holds.
 We note that $t_{L,1}\geq t_{L,2}\geq \dots\geq t_{L,d}$ and that %
our assumptions on the coefficients $c_0,\dots,c_d$ in \eqref{eq:C}   guarantee that
 \[
 \rho_{L,0}\abs{ z_{L,0}}\sqrt{ t_{L,1}} = {\pp{1-\frac\C{\sqrt L}}e^{-c_0/\sqrt L}}< 1 \qmand \frac{\rho_{L,1}}{\sqrt{t_{L,d}}}  = {\pp{1-\frac\A{\sqrt L}}e^{-c_d/\sqrt L}} < 1
 \]
 for $L$ large enough. We assume implicitly that $L$ is large enough so the above holds from now on.  Thus after some %
 rewriting
 we have
\begin{multline*}%
  \Phi_L \pp{\frac{\vv c}{\sqrt L},  \frac{\vv\theta}{\sqrt L}}       =
       \frac{\prod_{k=1}^d v_{L,k}^{-\sigma (L_k-L_{k-1})/(2+\sigma)}}{\mathfrak C_L}
\\       \times
     \EE\left[  \frac{ \prod_{k=1}^d(\sigma v_{L,k}+X_{t_{L,k}}/\sqrt{t_{L,k}})^{L_k-L_{k-1}}  }
     {(1-\rho_{L,0}  z_{L,0} X_{t_{L,1}}+\rho_{L,0} ^2 z_{L,0}^2t_{L,1})(1-\rho_{L,1} X_{t_{L,d}  }/t_{L,d}+\rho_{L,1}^2/t_{L,d}  )}\right].
\end{multline*}
\begin{lemma}%
 The normalizing constant satisfies:
   \begin{equation}\label{KLsim}
    \mathfrak C_L\sim
 (2+\sigma)^LL^{1/2}\cdot
       \frac{\sqrt{2 }}{\sqrt{2+\sigma}} \cdot \mathfrak C_{\A',\C'} \mbox{ as $L\to\infty$},
  \end{equation}
  where  $\mathfrak C_{\A,\C}$ is given by   \eqref{eq:C_ac} and $\A',\C'$ are from \eqref{a'c'}.
\end{lemma}
 \begin{proof}
We use the explicit form of expression \eqref{ZL-MP}: %
   \begin{align}
    \mathfrak C_L  &   =\frac{1}{2\pi}\int_{-2}^2
     \frac{(\sigma+ x)^L\sqrt{4-x^2}}{
          {(1-\rho_{L,0}   x+\rho_{L,0} ^2 )(1-\rho_{L,1}  x+\rho_{L,1}^2)}
     }dx+ \pp{\rho_{L,1}-\frac{1}{\rho_{L,1}}}_+ \frac{\pp{\rho_{L,1}+\frac1{\rho_{L,1}}+\sigma}^L}{(\rho_{L,1}-\rho_{L,0})(1-\rho_{L,0}\rho_{L,1})} \nonumber \\ %
&    =:I_L+D_L. \label{I+D}
\end{align}
\arxiv{Note that there is no atom if  $\A=\C$,  so  we put $D_L=0$ in this case. }
     The dominant term in the integral $I_L$ comes from the integral over $[0,2]$.
   This is easy to see  as
   $|x+\sigma|^L\leq \pp{max\{2-\sigma,\sigma\}}^L=o((2+\sigma)^L)$ for $-2\leq x\leq 0$ when $\sigma>0$.

     The argument for the asymptotics of the integral over $[0,2]$  relies on the substitution $x=2-u^2/L$ that appeared in a similar context in the paper \cite{Bryc-Wang-2015}  and later in \cite{Bryc-Wang-2017ASEP}, \cite{Bryc-Wang-Wesolowski-2022}.
  \begin{align*}
   I_L &   \sim  \frac{1}{2\pi}\int_{0}^2
     \frac{(\sigma+ x)^L\sqrt{4-x^2}}
          {(1-\rho_{L,0}   x+\rho_{L,0} ^2 )(1-\rho_{L,1}  x+\rho_{L,1}^2)} dx
         \\ &\sim  \frac{(2+\sigma)^L}{2\pi}\int_0^{\sqrt{2L}}
         \frac{ \pp{1-\frac{u^2}{(2+\sigma)L}}^L}{
         \left((\frac{\C}{\sqrt{L}})^2+(1-\frac{\C}{\sqrt{L}})\frac{u^2}{L}\right)\left((\frac{\A}{\sqrt{L}})^2+(1-\frac{\A}{\sqrt{L}})\frac{u^2}{L}\right)} \sqrt{4-\frac{u^2}{L}}\, \frac{2}{\sqrt{L}} u^2 \frac{d u}{L}\nonumber
         \\\
       &  \sim
         \frac{2}{\pi}(2+\sigma)^LL^{1/2}\int_0^{\sqrt{2L}}
         \frac{ \pp{1-\frac{u^2}{(2+\sigma)L}}^L}{
         \left(\C^2+u^2\right)
         \left(\A^2+u^2\right)}  u^2 d u\sim (2+\sigma)^LL^{1/2}\frac{2}{\pi}\int_0^{\infty}
         \frac{ u^2e^{-\frac{u^2}{2+\sigma}}du}{
         \left(\C^2+u^2\right)
         \left(\A^2+u^2\right)}.\nonumber
  \end{align*}
 The integral is an explicit expression \eqref{C-no-CC}; compare  \cite[(4.38) and Lemma A.2]{Bryc-Wang-Wesolowski-2022}.

 For $\A\ne\C$ we get
  \begin{equation}
   \label{I-sim}
   I_L\sim (2+\sigma)^L \sqrt{L} \cdot \frac{2}{\sqrt{2+\sigma}}\frac{|\C'| H(|\C'|/2) - |\A'| H(|\A'|/2)}{{\C'}^2-{\A'}^2}=(2+\sigma)^L \sqrt{L} \cdot \frac{\sqrt{2}}{\sqrt{2+\sigma}}\mathfrak C_{|\A'|,|\C'|}.
   \end{equation}
 This proves \eqref{KLsim} for $\A\geq 0$, $\A\ne \C$. For $\A=\C$ we get
 \begin{equation*}
   I_L\sim (2+\sigma)^L \sqrt{L} \cdot
\frac{\sqrt{\pi } \left(\A'^2+2\right)
   H\left(\frac{\A'}{2}\right)-2
   \A'}{2 \sqrt{\pi } \A'
   \sqrt{\sigma +2}}=(2+\sigma)^L \sqrt{L} \cdot \frac{\sqrt{2}}{\sqrt{2+\sigma}} \mathfrak C_{\A',\A'}.
 \end{equation*}
   Thus \eqref{KLsim} holds also  for $\A=\C>0$.

   When $\A<0,\C>0$ (but $\A+\C>0$) we need to include the contribution of the discrete part.
   It is easy to see that with $\rho_{L,1}=1-\A/\sqrt{L}>1$
   the discrete part in \eqref{I+D} is
   \begin{align*}
     D_L&=\frac{\A L \left(2 \sqrt{L}-\A\right)
   \left(-\frac{\A}{\sqrt{L}}+\frac{1}{1-\frac{\A}
   {\sqrt{L}}}+\sigma +1\right)^L}{(\A-\C)
   \left(\sqrt{L}-\A\right) \left(\sqrt{L}
   (\A+\C)-\A \C\right)}
   \sim \sqrt{L}(2+\sigma)^L \cdot
   \frac{2 \A
   \left(1+\frac{\A^2}{L (\sigma
   +2)}\right)^L}{\A^2-\C^2}\\
&   \sim
\sqrt{L}(2+\sigma)^L \cdot
   \frac{2 \A
   e^{\frac{\A^2}{\sigma
   +2}}}{\A^2-\C^2}= \sqrt{L}(2+\sigma)^L  \cdot \frac{\sqrt{2}}{\sqrt{2+\sigma}}
    \cdot
   \frac{2\sqrt{2} \A'
   e^{\frac{\A'^2}{4}}}{\A'^2-\C'^2}.\nonumber
   \end{align*}
   Combining this with \eqref{I-sim} we see that for $\A<0$, we have
   \[\mathfrak C_L\sim \sqrt{L}(2+\sigma)^L  \cdot \frac{\sqrt{2}}{\sqrt{2+\sigma}}
   \pp{\mathfrak C_{-\A',\C'}+  \frac{2\sqrt{2}\, \A'
   e^{\frac{\A'^2}{4}}}{\A'^2-\C'^2}}.
   \]
   We now use the  identity $\erfc(x)+\erfc(-x)=2 $ to verify that %
   $$\mathfrak C_{-\A',\C'}+   \frac{2\sqrt{2}\, \A'
   e^{\frac{\A'^2}{4}}}{\A'^2-\C'^2}= \mathfrak C_{\A',\C'}.$$
 This completes the proof.
  \end{proof}
It turns out that it suffices to restrict %
the expectation to %
the event
$\{X_{t_{L,k}}\ge 0\}$, as shown below.
\begin{lemma}%
If $\sigma>0$ then
    \begin{multline}\label{phi-re}
  \Phi_L \pp{\frac{\vv c}{\sqrt L},  \frac{\vv\theta}{\sqrt L}}
  \sim      \frac{\prod_{k=1}^d v_{L,k}^{-\sigma (L_k-L_{k-1})/(2+\sigma)}}{\mathfrak C_L}
  \\
\times     \EE\left[  \frac{ \prod_{k=1}^d\spp{(\sigma v_{L,k}+X_{t_{L,k}}/\sqrt{t_{L,k}})^{L_k-L_{k-1}}\inddd{X_{t_{L,k}}\geq 0}} }
     {(1-\rho_{L,0}  z_{L,0} X_{t_{L,1}}+\rho_{L,0} ^2 z_{L,0}^2t_{L,1})(1-\rho_{L,1} X_{t_{L,d} }/t_{L,d}+\rho_{L,1}^2/t_{L,d}  )}\right].
 \end{multline}
\end{lemma}
\begin{proof}Since $-2\sqrt{t}\leq X_t\leq 2 \sqrt{t}$ for $t>0$, $x_j-x_{j-1}>0$,  $v_{L,k}=1+O(L^{-1/2})$  and $t_
{L,d}\geq 1$, the expected value over the set $X_{t_{L,j}}<0$ is bounded by a factor
\begin{align*}
 e^{O(\sqrt{L})} \frac{(\max\{\sigma,2-\sigma\})^{L_j-L_{j-1}}(2+\sigma)^{L-(L_j-L_{j-1})}}
{(1-\rho_{L,0} z_{L,0}\sqrt{t_{L,1}})^2(1-\rho_{L,1}/\sqrt{t_{L,d}})^2}
& \sim e^{O(\sqrt{L})} (2+\sigma)^L L^2\pp{\frac{\max\{\sigma,2-\sigma\}}{2+\sigma}}^{L(x_j-x_{j-1})}\\
&=o\pp{(2+\sigma)^L \sqrt{L}} \mbox {as $L\to\infty$}
\end{align*}
Since $v_{L,k}=1+O(L^{-1/2})$, the factor in front of the expectation is  of  the order $e^{O(\sqrt{L})}$. By \eqref{KLsim} this proves \eqref{phi-re}.
 \arxiv{Indeed, noting that $s_d=c_d$, we have
 $$1-\rho_{L,1} \frac{X_{t_{L,d} }}{t_{L,d}}  +\rho_{L,1}^2/t_{L,d}\geq 1-2 \rho_{L,1}/\sqrt{t_{L,d}}+\rho_{L,1}^2/t_{L,d}=(1-\rho_{L,1}/\sqrt{t_{L,d}})^2\sim\pp{\frac{\A+c_d}{\sqrt{L}}}^2   \mbox{ as $L\to\infty$}.
 $$

 Similarly, noting that $s_0-s_1=c_0$, we have
 $$1-\rho_{L,0}  z_{L,0} X_{t_{L,1}}+\rho_{L,0} ^2 z_{L,0}^2t_{L,1}\geq 1-2\rho_{L,0}  z_{L,0} \sqrt{t_{L,1}}+\rho_{L,0} ^2 z_{L,0}^2t_{L,1} =(1-\rho_{L,0} z_{L,0}\sqrt{t_{L,1}})^2\sim \pp{\frac{\C+c_0}{\sqrt{L}}}^2 \mbox{ as $L\to\infty$}. $$
 }
\end{proof}

To determine the asymptotics of the expectation, introduce
\[
U_s := e^{-s}X_{e^{2s}}, s\in\R.
\]
This is a stationary $[-2,2]$-valued Markov process with univariate probabilities
\begin{equation}
  \label{pi:OU}
 \Pr(U_s\in dy)= \frac{\sqrt{4-y^2}}{2 \pi }\inddd{|y|\leq 2} dy
\end{equation}
and transition probabilities
\begin{multline}
  \label{pst:OU}
  \Pr(U_s\in dy|U_{s'}=y')\\
  =\frac{\sqrt{4-y^2}}{2 \pi }
  \frac{ e^{2 (s-s') }-1}{-2 y y'  \cosh (s-s' )+2 \cosh (2
   (s-s' ))+y^2+{y'}^2-2}dy ,
   \quad s'<s,\; y, y'\in[-2,2].
\end{multline}
So we arrive at%
\begin{multline}\label{phi2OU}
  \Phi_L \pp{\frac{\vv c}{\sqrt L},  \frac{\vv\theta}{\sqrt L}} \sim \frac{\prod_{k=1}^d v_{L,k}^{-\sigma (L_k-L_{k-1})/(2+\sigma)}}{\mathfrak C_L}
  \\
  \times
     \EE\left[  \frac{ \prod_{k=1}^d\spp{(\sigma v_{L,k}+U_{s_k/\sqrt{L}})^{L_k-L_{k-1}}\ind_{U_{s_k/\sqrt{L}}>0} }  }
     {(1-\rho_{L,0}  z_{L,0}\sqrt{t_{L,1}} U_{s_1/\sqrt{L}}+\rho_{L,0} ^2 z_{L,0}^2t_{L,1})(1-\rho_{L,1}  U_{s_d/\sqrt{L}}/\sqrt{t_{L,d}} +\rho_{L,1}^2/t_{L,d})}\right].
\end{multline}

Introduce
\[
Y_L(s):= L\pp{2-U_{s/\sqrt L}}.
\]
This is a well-studied Markov process, and we shall explain it later. Now, we re-write the expectation on the right-hand side of \eqref{phi2OU} as
\begin{multline*}
\prodd k1d(2+\sigma v_{L,k})^{L_k-L_{k-1}} \\ \times \EE\left[  \frac{ {\prod_{k=1}^d\spp{1 - \frac{Y_L(s_k)}{(2+\sigma v_{L,k})L}}^{L_k-L_{k-1}}   \inddd{Y_L(s_k)<2L}}  }
     {\spp{(1-\rho_{L,0}  z_{L,0}\sqrt{t_{L,1}})^2 +\rho_{L,0} z_{L,0} Y_L(s_1)/L
     }\times\spp{(1-\rho_{L,1}/\sqrt{t_{L,d}})^2  +\frac{\rho_{L,1} Y_L(c_d)}{\sqrt{t_{L,d}} L}}}\right].
\end{multline*}
Grouping the first product above with the first product on the right-hand side of \eqref{phi2OU}, we arrive at
  \begin{equation}
  \label{pre-fact}
  \Phi_L \pp{\frac{\vv c}{\sqrt L},  \frac{\vv\theta}{\sqrt L}} \sim \wt \psi_{L} \pp{\frac{\vv\theta}{L^{1/2}}}
 \psi_{L} \pp{\frac{\vv c}{L^{1/2}},\frac{\vv\theta}{L^{1/2}}},
\end{equation}
where
\begin{align}
 \wt  \psi_{L} \pp{\frac{\vv\theta}{L^{1/2}}} & =\prod_{k=1}^d\left(\frac{\sigma}{2+\sigma} v_{L,k}^{2/(2+\sigma)} +\frac{2}{2+\sigma}v_{L,k}^{-\sigma/(2+\sigma)}\right)^{L_k-L_{k-1}},\nonumber\\
\label{eq:psi_L}
\psi_{L} \pp{\frac{\vv c}{L^{1/2}},\frac{\vv\theta}{ L^{1/2}}}
& = \frac{L^2(2+\sigma)^L}{ \mathfrak C_L} \esp\bb{ G_L\pp{Y_L(s_1),\dots,Y_L(s_{d})}},\end{align}
with
\[
G_L(y_1,\dots,y_{d}):=\frac{\prodd j1{d}\spp{\spp{1-\frac{y_j}{(2+\sigma v_{L,j})L}}^{L_j-L_{j-1}} \inddd{|y_j|\le 2L}}}{\pp{L(1-\rho_{L,0}  z_{L,0}\sqrt{t_{L,1}})^2 +\rho_{L,0} z_{L,0} y_1}\pp{L(1-\rho_{L,1}/\sqrt{t_{L,d}})^2  +{\rho_{L,1} y_d}/{\sqrt{t_{L,d}}}}}.
\]
It is clear that
\begin{equation}
  \label{end0}
\lim_{L\to\infty}\wt \psi_{L} \pp{\frac{\vv\theta}{L^{1/2}}}= e^{\frac{\sigma}{(2+\sigma)^2}\sum_{k=1}^d  (x_k-x_{k-1})\wt s_k^2}.
\end{equation}
This determines the second factor on the right-hand side of \eqref{Lim-ph}.
For \eqref{eq:psi_L}, since $ v_{L,k}=1+\wt s_k/\sqrt{L}+O(L^{-1})$,
     \begin{align*}
     (L(1-\rho_{L,0}  z_{L,0}\sqrt{t_{L,1}})^2 +\rho_{L,0} z_{L,0} y_1) & =\left((\C+c_0)^2+y_1\right)+O(L^{-1/2}),\\     L(1-\rho_{L,1}/\sqrt{t_{L,d}})^2  +\rho_{L,1} y_d/\sqrt{t_{L,d}} & =\left((\A+c_d)^2+y_d\right)+O(L^{-1/2}),
     \end{align*}
we have
\[
\lim_{L\to\infty} G_L(y_1,\dots,y_{d})=G(y_1,\dots,y_{d}) :=\frac{ \exp\spp{-\summ j1{d}\frac{x_j-x_{j-1}}{2+\sigma}y_j}}{\left((\C+c_0)^2+y_1\right) \left((\A+c_d)^2+y_d\right)}.
\]
Let also $\pi_L(u)$ denote the marginal density of $Y_L(c_d)$. We have
\begin{align*}%
\psi_{L} \pp{\frac{\vv c}{L^{1/2}},\frac{\vv \theta}{ L^{1/2}}} & = \frac{L^2(2+\sigma)^L}{\mathfrak C_L} \esp\bb{ G_L\pp{Y_L(s_1),\dots,Y_L(s_{d})}}
\\ & \sim
\frac{L^{3/2}\sqrt{2+\sigma}}{ \sqrt{2}\mathfrak C_{\A',\C'}} \esp\bb{ G_L\pp{Y_L(s_1),\dots,Y_L(s_{d})}}\nonumber
 \\ & =\frac{\sqrt{2+\sigma}}{ \sqrt{2}\mathfrak C_{\A',\C'}}  \int _0^{2L} \esp\bb{G_L(Y_L(s_1),\dots,Y_L(s_{d-1}),u)\middle | Y_L(c_d)=u}\ L^{3/2}\pi_L(u)\,du.%
\end{align*}

Now we take a closer look at the process $\{Y_L(s)\}_{s>0}$.
This is a Markov process whose univariate density,   computed from \eqref{pi:OU}, is
\[%
  \Pr(Y_L(s)\in dv)=\pi_L(v)d v = \frac{\sqrt{v(4L-v)}}{2\pi L^2} \mathbf{1}_{0<v<4L}  dv;
\]%
 compare %
\cite[Lemma 4.2]{Bryc-Wang-2017ASEP}. Its transition probabilities for $s_k>s_{k+1}$ that can be computed from \eqref{pst:OU}. Moreover, it is known \citep{Bryc-Wang-2015} that as $L\to\infty$,
\[%
\calL\pp{\pp{Y_L(s)}_{s\ge c_d}\mmid  Y_L(c_d) = u}  \fddto \calL\pp{\pp{\zeta_{s}}_{s\ge c_d}\mmid \zeta_{c_d} = u},
\]%
where we let $\zeta$ denote the Markov process with  transition probabilities $\Pr(\zeta_t\in dy|\zeta_s=x)=\p_{t-s}(x,y)dy$ given in
 \eqref{biane-p}.
In the above, $\calL(\cdot\mid\cdot)$ is understood as the law induced by the conditional law of the corresponding Markov process starting at fixed time from a fixed point $u$.

In view of the bound  $\pp{1-y/((2+v\sigma)L)}^{x L}\leq \exp(-xy/(2+v\sigma))$ which is valid for $0\leq y\leq 2L$, $x>0$, $v>0$,  we see that
\[G_L(y_1,\dots,y_d) \leq \frac{C}{(\C+c_0)^2(\A+c_d)^2}\exp\pp{-\sum_{k=1}^d (x_k-x_{k-1})y_k/(4+3\sigma)}\]
for some $C$ and large $L$. (More precisely, there is $L_0$ and $C$ such that this bound holds for all $L\geq L_0$ and all $0\leq y_k\leq 2L$, but the bound extends to all $0\leq y_k<\infty$ as $G_L(\vv y)=0$ when some $y_k>2L$.)
\arxiv{ To use the exponential bound, we note that $G_L(\vv y)$ contains a product of factors
$\left(1-\frac{y_k}{(2+\sigma v_{L,k})L}\right)^{L_k-L_{k-1}}$. For each factor, we use the exponential bound with $x=(L_{k}-L_{k-1})/L$ and $v=v_{L,k}$, which shows that for $0\leq y_k\leq 2L$ each factor is bounded by
$\exp\left( - y_k  \frac{L_k-L_{k-1}}{L(2+\sigma v_{L,k})}\right)$. However, $v_{L,k}=1+O(L^{-1/2})$ and $(L_{k}-L_{k-1})/L\to x_k-x_{k-1}>0$, so for large $L$,  each factor is bounded by $\exp\left(- y_k\tfrac{x_k-x_{k-1}}{2} \tfrac{1}{2+3\sigma/2}\right)$.
}

Thus, by invoking  \cite[Exercise 6.6]{billingsley99convergence} or applying the dominated convergence theorem, we  see that
\[
\lim_{L\to\infty}\psi_{L} \pp{\frac{\vv c}{L^{1/2}}, \frac{\vv\theta}{L^{1/2}}}
=\frac{\sqrt{2+\sigma}}{ \sqrt{2}\pi \mathfrak C_{\A',\C'}}  \int _0^{\infty}
\esp\bb{G(\zeta_{s_1},\dots,\zeta_{s_d})\middle |\zeta_{s_d}=u} \sqrt{u}\, du =\Psi(\vv c).
\]
\arxiv
{Note that the elementary reversibility identity $\sqrt{x}\p_t(x,y)=\sqrt{y}\p_t(y,x)$
converts the expression $\sqrt{u_d}\prod_{k=1}^{d-1}\p_{c_k}(u_{k+1},u_k)$ that appears as the limit in the above formula into the expression   defining $\Psi(\vv c)$ in  \eqref{lim-ps}.

}
Combined with \eqref{pre-fact} and \eqref{end0}, this completes the proof of \eqref{Lim-ph}.

\section{Proof of Proposition \ref{prop:1}}\label{sec:prop1}
To avoid cumbersome notation and additional technicalities, we prove Proposition \ref{prop:1} for
\begin{equation}
  \label{rho(L)+}
 \rho_{L,0}=e^{-\C/\sqrt{L}},\;  \rho_{L,1}=e^{-\A/\sqrt{L}}
\end{equation}
instead of the asymptotically equivalent expression \eqref{rho(L)}.

First, it is known, see \cite{barraquand2022steady}, that the law $\Pr_{\eta\topp{\A,\C}/\sqrt{2}}$ on $C[0,1]$ of process $\eta\topp{\A,\C}/\sqrt{2}$   is absolutely continuous with respect to the law $\Pr_{B}$ of the Brownian  motion   of variance 1/2  with the Radon-Nikodym derivative
     \begin{equation}\label{BD-21}\frac{d \Pr_{\eta\topp{\A,\C}/\sqrt{2}}}{d \Pr_B} = \frac{2}{(\A+\C)\mathfrak C_{\A,\C}} e^{ (\A+\C)\min_{x\in[0,1]}B_x-  \A  B_1 },
     \end{equation}
     where  $\mathfrak C_{\A,\C}$ is the normalizing constant \eqref{C(a,c)}. %

Next,  denote by $\vv S=\{S_1,S_2,\dots,S_L\}$ a random walk starting at $0$ with   i.i.d. increments taking values in $\{\pm 1,0\}$ with probabilities  $1/(2+\sigma)$ and $\sigma/(2+\sigma)$ respectively. Introduce its partial-sum process
\begin{equation}
  \label{zeta}
  \zeta_L(t):=S_{\floor{Lt}}, t\in[0,1].
\end{equation}

The law of $\xi_L$ is absolutely continuous with respect to the law of $\zeta_L$ on $D[0,1]$, denoted by $\Pr_{\xi_L}, \Pr_{\zeta_L}$ respectively. To see this, it suffices to compare the laws of the vectors $\vv \gamma^\circ_L=\{\gamma_k-\gamma_0\}_{k=1,\dots, L}$  and $\vv S_L = (S_1,\dots,S_L)$ on $\ZZ^L$, denoted by $\Pr_{\vv\gamma^\circ_L}, \Pr_{\vv S_L}$ respectively.
 By summing over the values of $\gamma_0\in\ZZ_{\geq 0}$ such that $\min_{0\leq k\leq L}\{\gamma_k\}\geq 0$, for $\vv s=\{s_1,\dots,s_L\}$ in the support of $\vv S_L$,  the Radon-Nikodym derivative is
\[   \frac{d\Pr_{\vv \gamma^\circ_L}}{d\Pr_{\vv S_L}}( \vv s)
    =  \frac{1}{\mathsf C_L} (\rho_{L,0} \rho_{L,1})^{-\min_{k=0,\dots,L} s_k} \rho_{L,1}^{s_L},
\]
 where ${\mathsf C}_L$ is the normalizing constant and $s_0 = 0$. It then follows that
 with $\omega = \{\omega_t\}_{t\in[0,1]}\in D[0,1]$ we have
 \begin{equation}
    \label{RN-gamma}
 \frac{d\Pr_{\xi_L}}{d\Pr_{\zeta_L}}(\omega) = \frac1{\mathsf %
 C_L}(\rho_{L,0}\rho_{L,1})^{-\min_{x\in[0,1]}\omega_x}\rho_{L,1}^{\omega_1}=\frac1{\mathsf C_L}
 \calE(\omega/\sqrt{L}),
 \end{equation}
where we used \eqref{rho(L)+} and denoted $\calE(\omega):=\exp\pp{(\A+\C)\inf_{x\in[0,1]}\omega_x -{\A} \omega_1}$.
 Formula \eqref{RN-gamma} implies that for any bounded continuous function $\Phi:D[0,1]\to\RR$ we have
\begin{equation}
  \label{PhiE}
     \EE\left[\Phi\left(\frac{\xi_L}{\sqrt L}\right)\right]
   = \frac{1}{\mathsf C_L}\EE\bb{\Phi\left(\frac{\zeta_L}{\sqrt{L}}\right)\calE\pp{\frac{\zeta_L}{\sqrt L}}},
\end{equation}
Since the increments of $\vv S$ have mean zero and variance $2/(2+\sigma)$, by Donsker's theorem
$$\frac1{\sqrt{L}}\{\zeta_L(x)\}_{x\in[0,1]}\Rightarrow \frac2{\sqrt {2+\sigma}}\{B_x\}_{x\in[0,1]}$$ in  %
$D[0,1]$.
Since %
$\sup_{L=1,2,\dots} \EE\left[\calE\pp{\frac{\zeta_L}{\sqrt L}}^2\right]<\infty$
and $\Phi$ is bounded, the real random variables $\Phi\pp{\frac{\zeta_L}{\sqrt{L}}}\calE\pp{\frac{\zeta_L}{\sqrt L}}$, $L=1,2,\dots$ are uniformly integrable. Uniform integrability and weak convergence imply convergence of expectations (\cite[Theorem 3.5]{billingsley99convergence}),
 so it follows that
 \begin{align}
  \lim_{L\to\infty}\EE\left[\Phi\pp{\frac{\zeta_L}{\sqrt{L}}}\calE\pp{\frac{\zeta_L}{\sqrt L}}\right] & = \EE\left[\Phi\pp{\frac2{\sqrt{2+\sigma}}B}\calE\pp{\frac2{\sqrt{2+\sigma}}B}\right]\nonumber\\
& = \EE\bb{\Phi\pp{\frac2{\sqrt{2+\sigma}}B}e^{(\A+\C)\inf_{x\in[0,1]}\frac2{\sqrt{2+\sigma}}B_x-\A\frac2{\sqrt{2+\sigma}}B_1}}\nonumber\\
&  = \EE\left[\Phi\pp{\frac2{\sqrt{2+\sigma}}B} e^{(\A'+\C')\min_{0\leq x\leq 1}B_x-\A'  B_1 }\right].\label{Don3}
 \end{align}
  In particular,   \eqref{Don3} with $\Phi\equiv 1$ implies that the normalizing constants converge,
  $\mathsf C_L\to  (\A'+\C')\mathfrak C_{\A',\C'}/2 $. %
Dividing \eqref{Don3} by these normalizing constants%
and using formulas  \eqref{PhiE} and \eqref{BD-21}
we get
$$
\lim_{L\to\infty} \EE\left[\Phi\left(\frac{\xi_L}{\sqrt{L}}\right)\right]
=\EE\bb{\Phi\pp{\sqrt{\frac2{2+\sigma}}\eta \topp{\A',\C'}}},
$$
for all continuous and bounded functions $\Phi$ from $D[0,1]$ to $\R$. This completes the proof of \eqref{BD-Don} under assumption \eqref{rho(L)+}. We omit the proof  of \eqref{BD-Don} under assumption \eqref{rho(L)}, as it requires cumbersome notation and  additional steps.
\arxiv{ %
We sketch the omitted argument for the expanded arxiv version of the paper.
Define
\begin{equation}\label{8}
 \A_L=-\sqrt L\log(1-\A/\sqrt L),\qquad
 \C_L=-\sqrt L\log(1-\C/\sqrt L).
\end{equation}
Then    \eqref{rho(L)} is \emph{exactly} the exponential parametrization \eqref{rho(L)+} with
\((\A,\C)\) replaced by \((\A_L,\C_L)\), and
\((\A_L,\C_L)\to(\A,\C)\).  The   Radon--Nikodym calculation can therefore
be repeated with varying parameters.  The only additional point is uniform
integrability.  Since \(\A_L,\C_L\) are bounded for large \(L\),
\begin{equation}\label{9}
 \left|(\A_L+\C_L)\min_{k\leq L}S_k-\A_LS_L\right|
 \leq K\max_{k\leq L}|S_k|.
\end{equation}
Standard sub-Gaussian maximal estimates for the bounded-increment walk give
uniform exponential moments for the right side of \eqref{9} after division by
\(\sqrt L\).  Donsker's theorem, continuity in \(\A_L,\C_L\), and the same
uniform-integrability argument then give the stated limit.

}

\subsection*{Acknowledgement}
We thank two anonymous referees for their helpful comments that have improved the paper. WB's research was partially supported by Simons Foundation  Award Number: 703475, US.
YW's research was partially supported by Army Research Office, US (W911NF-20-1-0139).
Both authors acknowledge support from the  Charles Phelps Taft Research Center at the University of Cincinnati.

\arxiv{
\subsection*{Use of AI tools declaration}
The published version of this paper did not rely on AI tools. This arXiv resubmission benefited from a report  produced with GPT Sol 5.6 Max. The report identified numerous typographical errors, including missing factors in \eqref{Mkv-ans0}, use of implicit assumption $\bfB=\sigma \bfI$ in Proposition~\ref{prop:markov} and implicit assumption $d\geq 2$ in Proposition~\ref{prop:duality}. After verification by the authors, the typographical errors were corrected, the statement of Proposition~\ref{prop:markov} was amended, and we added a sketch of the proof of Proposition~\ref{prop:1} under the stronger assumption \eqref{rho(L)+} in place of the assumptions \eqref{rho(L)}.

}
\appendix
\section{Auxiliary formulas}
The following integral is known; for a derivation, see for example \cite[Lemma A.2]{Bryc-Wang-Wesolowski-2022}.
\begin{lemma}
The normalizing constant
\[%
 \mathfrak C_{\A,\C}= \int_{\RR_+^2} e^{-(\C x+\A y)/{\sqrt 2}}\g_1(x,y)dx dy
\]%
 is given by the expression
   \begin{equation}\label{eq:C_ac}
\mathfrak C_{\A,\C}
= \begin{cases}
\displaystyle {\sqrt 2}\cdot \frac{\A H(\A/2)-\C H(\C/2)}{\A^2-\C^2},& \mbox{ if } \A\ne\C, \; \A+\C>0,\\\\
\displaystyle \frac{2+\A^2 }{2\sqrt 2\A}\cdot H(\A/2) - \frac1{\sqrt{2\pi}}, & \mbox{ if } \A = \C>0,
\end{cases}
\end{equation}
  where for $x\in\RR$,
\[%
H(x) =e^{x^2}\erfc(x) \qmwith \erfc (x) = \frac 2{\sqrt\pi}\int_x^\infty e^{-t^2}dt.
\]%

\end{lemma}
The following is a minor re-write of known integrals; see \cite[Lemma 4.5]{Bryc-Kuznetsov-2021} or \cite[formula after  (4.38)]{Bryc-Wang-Wesolowski-2022}.
\begin{lemma} For $\A+\C>0$ and $\tau>0$, we have
\begin{equation}\label{C-no-CC}
    \frac{1}{2\pi}\int_0^\infty e^{-\tau v^2/2} \frac{4 v^2}{(\A^2+ v^2) (\C^2+v^2)}d v  = \sqrt{\tau}\mathfrak C_{|\A|\sqrt{2\tau},|\C|\sqrt{2\tau}}.
\end{equation}
\end{lemma}
\arxiv{
\begin{proof} For completeness, we include the elementary derivation based on the Laplace transform of the sine function as in the proof of \cite[Lemma 4.5]{Bryc-Kuznetsov-2021}.
   For non-zero $\A,\C$ the integral in \eqref{C-no-CC} is
   \begin{multline*}
      \frac{1}{2\pi}\int_0^\infty e^{-\tau v^2/2} \frac{4 v^2}{(\A^2+ v^2) (\C^2+v^2)}d v =
       \frac{2}{\pi}\int_0^\infty e^{-\tau v^2/2} \pp{\int_0^\infty e^{-|\A| x}\sin (v x) d x} \pp{\int_0^\infty e^{-|\C| y}  \sin (v y)d y} d v
      \\=   \int_0^\infty \int_0^\infty e^{-|\A| x-|\C| y} \pp{ \frac{2}{\pi}\int_0^\infty e^{-\tau v^2/2} \sin (v x)   \sin (v y)d v } d x d y =     \int_0^\infty \int_0^\infty  e^{-|\A| x-|\C| y}  \g_\tau(x,y)  d x d y\\=  \sqrt{\tau}\mathfrak C_{|\A|\sqrt{2\tau},|\C|\sqrt{2\tau}},
   \end{multline*}
  where we used  \eqref{trig2g}, Fubini's theorem,  and scaling $\g_{\tau x}(z,z')=\frac{1}{\sqrt{\tau}}\g_x(z/\sqrt{\tau},z'/\sqrt{\tau})$. By taking a limit (or modifying the above  calculation) the formula extends to $\A=0$ or $\C=0$.
\end{proof}

}

\def\polhk#1{\setbox0=\hbox{#1}{\ooalign{\hidewidth
  \lower1.5ex\hbox{`}\hidewidth\crcr\unhbox0}}}

\end{document}